\documentclass[12pt]{amsart}
\usepackage{amsmath,amssymb,bm,graphicx}
\usepackage{subfigure}
\theoremstyle{plain}
\newtheorem{theorem}{Theorem}[section]
\newtheorem{lemma}[theorem]{Lemma}

\theoremstyle{remark}
\newtheorem{remark}[theorem]{Remark}
\begin{document}
\allowdisplaybreaks[4]
\numberwithin{equation}{section} \numberwithin{figure}{section}
\numberwithin{table}{section}
\def\O{\Omega}
\def\p{\partial}
\def\R{\mathbb{R}}
\def\argmin{\mathop{\rm argmin}}
\def\ES{H^2_0(\O)}
\def\cT{\mathcal{T}}
\def\cE{\mathcal{E}}
\def\cV{\mathcal{V}}
\def\cC{\mathcal{C}}
\def\TSum{\sum_{T\in\cT_h}}
\def\ESum{\sum_{e\in\cE_h}}
\def\VSum{\sum_{p\in\cV_h}}
\def\ESumI{\sum_{e\in\cE_h^i}}
\def\Mean#1{\Big\{\hspace{-5pt}\Big\{\frac{\p^2 #1}{\p n^2}\Big\}\hspace{-5pt}\Big\}}
\def\Jump#1{\Big[\hspace{-3pt}\Big[\frac{\p #1}{\p n}\Big]\hspace{-3pt}\Big]}
\def\jump#1{[\hspace{-1pt}[{\p #1}/{\p n}]\hspace{-1pt}]}
\def\jumpTwo#1{[\hspace{-1pt}[{\p^2 #1}/{\p n^2}]\hspace{-1pt}]}
\def\jumpThree#1{[\hspace{-1pt}[{\p^3 #1}/{\p n^3}]\hspace{-1pt}]}
\def\HTh{H^2(\O;\cT_h)}
\def\Osch{\mathrm{Osc}_h}
\def\hT{h_{\scriptscriptstyle T}}
\def\hp{h_p}
\def\etaT{\eta_{\scriptscriptstyle T}}
\def\etaeO{\eta_{e,{\scriptscriptstyle 1}}}
\def\etaeTwo{\eta_{e,{\scriptscriptstyle 2}}}
\def\etaeThree{\eta_{e,{\scriptscriptstyle 3}}}
\def\zT{z_{\scriptscriptstyle T}}
\def\OSC{{\mathrm{Osc}(f;\cT_h)}}
\def\Osc{{\mathrm{Osc}}}
\def\tK{\tilde{K}}
\def\tu{\tilde{u}}
\def\tl{\tilde{\lambda}}
\def\d{\displaystyle}
\def\bT{b_{\scriptscriptstyle T}}
\def\CS{C_{\scriptscriptstyle \rm Sobolev}}
\def\CPF{C_{\scriptscriptstyle \rm PF}}
\def\tbar{|\!|\!|}
\def\bfT{\bar f_{\scriptscriptstyle T}}
 \title[{\em A Posteriori} Analysis for the Obstacle Problem of Kirchhoff Plates]
 {An {\em A Posteriori} Analysis of $\bm C^{\bf 0}$ Interior Penalty  Methods
 for the Obstacle Problem of Clamped Kirchhoff Plates}
\author[S.C. Brenner]{Susanne C. Brenner}
\address{Susanne C. Brenner, Department of Mathematics and Center for
Computation and Technology, Louisiana State University, Baton Rouge,
LA 70803} \email{brenner@math.lsu.edu}
\author[J. Gedicke]{Joscha Gedicke}
\address{Joscha Gedicke, Department of Mathematics and Center for
Computation and Technology, Louisiana State University, Baton Rouge,
LA 70803} \email{jgedicke@math.lsu.edu}
\author[L.-Y. Sung]{Li-yeng Sung} \address{Li-yeng Sung,
 Department of Mathematics and Center for Computation and Technology,
 Louisiana State University, Baton Rouge, LA 70803}
\email{sung@math.lsu.edu}
\author[Y. Zhang]{Yi Zhang}
\address{Yi Zhang, Department of Mathematics, University of Tennessee,
Knoxville, TN 37996}\email{yzhan112@utk.edu}
%\date{September 11, 2015}
\begin{abstract}
  We develop an {\em a posteriori} analysis of $C^0$ interior penalty methods
  for the displacement obstacle problem of clamped Kirchhoff plates.  We show that
  a residual based error estimator originally designed for $C^0$ interior
  penalty methods for the boundary value problem of clamped Kirchhoff plates can
  also be used for the obstacle problem.  We obtain reliability and efficiency
  estimates for the error estimator and introduce an adaptive algorithm based on
  this error estimator.  Numerical results indicate that the performance of the
  adaptive algorithm is optimal for both quadratic and cubic $C^0$ interior
  penalty methods.
\end{abstract}
\thanks{The work of the first and third authors was supported in part
 by the National Science Foundation under Grant No.
 DMS-13-19172.  The work of the second author was supported in part by a fellowship
 within the Postdoc-Program of the German Academic Exchange Service (DAAD)}
\maketitle
\section{Introduction}\label{sec:Introduction}
 Let $\O\subset\R^2$ be a bounded polygonal domain, $f\in L_2(\O)$, $\psi\in C(\bar\O)\cap C^2(\O)$ and
 $\psi<0$ on $\p\O$.
 The displacement obstacle problem for the clamped Kirchhoff plate is to find
\begin{equation}\label{eq:Obstacle}
  u=\argmin_{v\in K}\Big[\frac12 a(v,v)-(f,v)\Big]
\end{equation}
 where
\begin{equation}\label{eq:BilinearForms}
  a(w,v)=\int_\O D^2w:D^2 v\,dx=\int_\O \sum_{i,j=1}^2
  \Big(\frac{\p^2 w}{\p x_i\p x_j}\Big)\Big(\frac{\p^2 v}{\p x_i\p x_j}\Big)
   dx, \quad (f,v)=\int_\O fv\,dx
\end{equation}
 and
\begin{equation}\label{eq:KDef}
  K=\{v\in H^2_0(\O):\,v\geq\psi\quad\text{in}\quad \O\}.
\end{equation}
\par
 The unique solution $u\in K$ of \eqref{eq:Obstacle}--\eqref{eq:KDef}
 is characterized by the variational inequality
\begin{equation*}%\label{eq:VI}
  a(u,v-u)\geq (f,v-u)\qquad\forall\,v\in K,
\end{equation*}
 which can be written in the following equivalent complementarity form:
\begin{equation}\label{eq:ComplementarityCondition}
 \int_\O (u-\psi)\,d\lambda=0,
\end{equation}
 where the Lagrange multiplier $\lambda$ is the nonnegative Borel measure defined by
\begin{equation}\label{eq:lambdaDef}
  a(u,v)=(f,v)+\int_\O v\,d\lambda  \qquad \forall\,v\in H^2_0(\O).
\end{equation}
\begin{remark}\label{rem:SupportOfLambda}
  Since $u>\psi$ near $\p\O$, the support of $\lambda$ is disjoint from $\p\O$ because of
  \eqref{eq:ComplementarityCondition}.
\end{remark}
\begin{remark}\label{rem:lambda}
 We can treat $\lambda$ as a member of $H^{-2}(\O)=[H^2_0(\O)]'$ such that
\begin{equation*}
  \langle\lambda,v\rangle=\int_\O v\,d\lambda \qquad \forall\,v\in H^2_0(\O).
\end{equation*}
\end{remark}
\par
 $C^0$ interior penalty methods
 \cite{EGHLMT:2002:DG3D,BSung:2005:DG4,BNeilan:2011:SingularPerturbation,
 BGGS:2012:CH,Brenner:2012:C0IP,GGN:2013:C0IP}
 form a natural hierarchy of
  discontinuous Galerkin methods that are proven to be effective
  for fourth order elliptic boundary value problems.
 The goal of this paper is to develop an {\em a posteriori} error analysis of
 $C^0$ interior penalty methods
 for the obstacle problem defined by \eqref{eq:Obstacle}--\eqref{eq:KDef}.
 While there is a substantial literature on the {\em a posteriori} error analysis of finite
 element methods for second order
 obstacle problems (cf.
\cite{HK:1994:MGObstacle,CN:2000:Obstacle,Veeser:2001:Obstacle,NSV:2003:Obstacle,BC:2004:Obstacle,
NSV:2005:Contact,SV:2007:Adaptive,
BHS:2008:Obstacle, BCH:2009:Obstacle,GP:2014:Obstacle,GP:2015:Quadratic,CH:2015:Obstacle}
and the references therein),
 as far as we know this is the first paper on the {\em a posteriori} error analysis
 for the displacement obstacle problem of Kirchhoff plates.
 We note that there is a
 fundamental difference between second order and fourth order obstacle problems, namely
 that the Lagrange multipliers for the fourth order
 discrete obstacle problems can be represented naturally as  sums of Dirac point measures
 (cf. Section~\ref{sec:C0IP}), which
 leads to a simpler {\em a posteriori} error analysis (cf. Section~\ref{sec:Reliability} and
 Section~\ref{sec:Efficiency}).
\par
 The rest of the paper is organized as follows. We recall the $C^0$ interior penalty
 methods in Section~\ref{sec:C0IP} and analyze a mesh-dependent boundary value problem
 in Section~\ref{sec:BVP} that plays an important role in the {\em a posteriori} error analysis
 carried out in Section~\ref{sec:Reliability} and Section~\ref{sec:Efficiency}.
 An adaptive algorithm motivated by the {\em a posteriori} error analysis is
 introduced in Section~\ref{sec:Adaptive} and we report results of several
 numerical experiments
 in Section~\ref{sec:Numerics}.  We end the paper with some concluding remarks in Section~\ref{sec:Conclusions}.
\section{$C^0$ Interior Penalty Methods}\label{sec:C0IP}
 Let $\cT_h$ be a triangulation of $\O$, $\cV_h$ be the set of the vertices of $\cT_h$,
 $\cE_h$ be the set of the edges of $\cT_h$, and $V_h\subset H^1_0(\O)$
 be the $P_k$ Lagrange finite element space ($k\geq2$)
 associated with $\cT_h$.  The discrete problem for the $C^0$ interior
 penalty method \cite{BSZZ:2012:Kirchhoff,BSZ:2012:Obstacle} is to find
\begin{equation}\label{eq:C0IP}
   u_h=\argmin_{v\in K_h}\Big[\frac12 a_h(v,v)-(f,v)\Big],
\end{equation}
 where $K_h=\{v\in V_h:\, v(p)\geq\psi(p)\quad\text{for all}\;p\in \cV_h\}$,
\begin{align*}
 a_h(w,v)&=\TSum \int_TD^2w:D^2v\,dx+\ESum\int_e\Big(\Mean{w}\Jump{v} +\Mean{v}\Jump{w}\Big)ds\\
    &\hspace{40pt}+\ESum \frac{\sigma}{|e|}\int_e\Jump{w}\Jump{v}\,ds,
\end{align*}
 $\{\hspace{-3.5pt}\{\cdot\}\hspace{-3.5pt}\}$ denotes the average across an edge,
 $[\hspace{-1.5pt}[\cdot]\hspace{-1.5pt}]$ denotes the jump across an edge, $|e|$ is the length of
 the edge $e$,
 and $\sigma\geq1$ is a penalty parameter large enough so that $a_h(\cdot,\cdot)$ is positive-definite
  on $V_h$.  Details for the notation and the choice of $\sigma$ can be found in
  \cite{BSung:2005:DG4,JSY:2014:C0IP}.
\par
  The unique solution $u_h\in K_h$ of \eqref{eq:C0IP} is characterized by the variational
 inequality
\begin{equation*}
  a_h(u_h,v-u_h)\geq (f,v-u_h)\qquad\forall\,v\in K_h,
\end{equation*}
 which can be expressed in the following equivalent complementarity form:
\begin{equation}\label{eq:DiscreteComplentarityCondition}
 \sum_{p\in\cV_h}\lambda_h(p)\big(u_h(p)-\psi(p)\big)=0,
\end{equation}
 where the Lagrange multipliers $\lambda_h(p)$ are defined by
\begin{equation}\label{eq:lambdahDef}
  a_h(u_h,v)=(f,v)+\sum_{p\in\cV_h}\lambda_h(p)v(p)\qquad\forall\,v\in V_h
\end{equation}
 and satisfy
\begin{equation}\label{eq:SignOflambdah}
  \lambda_h(p)\geq0 \qquad\forall\,p\in\cV_h.
\end{equation}
\par
 We also use $\lambda_h$ to denote the measure $\sum_{p\in\cV_h}\lambda_h(p)\delta_p$,
  where $\delta_p$ is the Dirac point measure at $p$.  The equation \eqref{eq:lambdahDef} can therefore
   be written as
\begin{equation}\label{eq:AlternativelambdahDef}
  a_h(u_h,v)=(f,v)+\int_\O v\,d\lambda_h\qquad\forall\,v\in V_h.
\end{equation}
\begin{remark}\label{rem:2And4}
  For second order obstacle problems, the discrete Lagrange multiplier cannot be extended to
  $H^{-1}(\O)$ as a sum of Dirac point measures since such measures do not belong to $H^{-1}(\O)$.
  Consequently there are different choices for extending the discrete Lagrange multiplier
  to $H^{-1}(\O)$ \cite{Veeser:2001:Obstacle,NSV:2003:Obstacle,NSV:2005:Contact}.
  The fact that the Lagrange multiplier for the discrete fourth order obstacle problem can be expressed
  naturally as a sum of Dirac point measures leads to the simple {\em a posteriori}
  error analysis in Section~\ref{sec:Reliability} and
   Section~\ref{sec:Efficiency}.
\end{remark}
\begin{remark}\label{rem:lambdah}
 We can also treat $\lambda_h$ as a member of $H^{-2}(\O)=[H^2_0(\O)]'$ such that
\begin{equation*}
  \langle\lambda_h,v\rangle=\int_\O v\,d\lambda_h=\sum_{p\in\cV_p}\lambda_h(p)v(p) \qquad\forall\,v\in H^2_0(\O).
\end{equation*}
\end{remark}
\par
 Let the mesh-dependent norm $\|\cdot\|_h$ be defined by
\begin{equation}\label{eq:hNormDef}
 \|v\|_h^2=\TSum |v|_{H^2(T)}^2+\ESum\frac{\sigma}{|e|}\|\jump{v}\|_{L_2(e)}^2.
\end{equation}
 Note that
\begin{equation}\label{eq:SameNorm}
  \|v\|_h=|v|_{H^2(\O)} \qquad\forall\,v\in \ES.
\end{equation}
\par
 The following {\em a priori} error estimate is known \cite{BSZ:2012:Obstacle,BSZZ:2012:Kirchhoff}:
\begin{equation}\label{eq:APriori}
  \|u-u_h\|_h\leq Ch^\alpha,
\end{equation}
 where the index of elliptic regularity $\alpha\in(\frac12,1]$ is determined by the interior angles
 of $\O$ and can be taken to be $1$ if $\O$ is convex.
\par
 Our goal is to develop {\em a posteriori} error estimates for $\|u-u_h\|_h$.
\par
 Two useful tools for the analysis of $C^0$ interior penalty methods are the
 nodal interpolation operator $\Pi_h:\ES\longrightarrow V_h$ and an enriching operator
 $E_h:V_h\longrightarrow W_h\subset H^2_0(\O)$, where $W_h$ is the Hsieh-Clough-Tocher macro finite element space \cite{Ciarlet:1974:HCT}.
 \par
   The operator $E_h$ is defined by averaging
 (cf. \cite[Section~4.1]{Brenner:2012:C0IP})
 and hence
\begin{equation}\label{eq:Invariance}
  (E_hu_h)(p)=u_h(p)\quad\text{for all}\; p\in\cV_h.
\end{equation}
 The following estimate can be found in the proof of \cite[Lemma~1]{Brenner:2012:C0IP}.
\begin{equation}\label{eq:EhFundamentalEstimate}
  \hT^{-4}\|v-E_hv\|_{L_2(T)}^2\leq C\sum_{e\in\tilde\cE_T}
   \frac{1}{|e|}\|\jump{v}\|_{L_2(e)}^2 \qquad\forall\,T\in\cT_h,
\end{equation}
 where $\tilde\cE_T$ is the set of the edges of $\cT_h$ emanating from the vertices of $T$, and
 the positive constant $C$ depends only on $k$ and the shape regularity of $\cT_h$.
\par
 From \eqref{eq:EhFundamentalEstimate} and standard inverse estimates \cite{Ciarlet:1978:FEM,BScott:2008:FEM}, we also have
\begin{alignat}{3}
  \hT^{-2}\|v-E_hv\|_{L_\infty(T)}^2&\leq C\sum_{e\in\tilde\cE_T}
   \frac{1}{|e|}\|\jump{v}\|_{L_2(e)}^2 &\qquad&\forall\,T\in\cT_h,\label{eq:EhLInfty}\\
  \TSum|v-E_hv|_{H^2(T)}^2&\leq C\ESum\frac{1}{|e|}\|\jump{v}\|_{L_2(e)}^2&\qquad&\forall\,v\in V_h,
   \label{eq:EhEnergyEst}\\
   \|v-E_hv\|_h^2&\leq C\ESum\frac{\sigma}{|e|}\|\jump{v}\|_{L_2(e)}^2&\qquad&\forall\,v\in V_h,
   \label{eq:EhhNormEst}
\end{alignat}
 where the positive constant $C$ depends only on $k$ and the shape regularity of $\cT_h$.
\section{A Mesh-Dependent Boundary Value Problem}\label{sec:BVP}
  Let $z_h\in H^2_0(\O)$ be defined by
\begin{equation}\label{eq:zhDef}
  a(z_h,v)=(f,v)+\int_\O v\,d\lambda_h=(f,v)+\sum_{p\in\cV_h}
  \lambda_h(p)v(p)\qquad\forall\,v\in H^2_0(\O).
\end{equation}
 Then $u_h$ is the approximate solution of \eqref{eq:zhDef}
  obtained by the $C^0$ interior penalty method.
\begin{remark}\label{rem:Braess}
  The idea of considering such mesh-dependent boundary value problems was introduced in
 \cite{Braess:2005:Obstacle} for second order obstacle problems.
\end{remark}
\par
 A residual based error estimator \cite{BGS:2010:C0IP,Brenner:2012:C0IP} for $u_h$  (as an approximate solution of \eqref{eq:zhDef})
 is given by
\begin{equation}\label{eq:BVPEstimator}
 \eta_h=\Big(\ESum\etaeO^2+\ESumI (\etaeTwo^2+\etaeThree^2)+\TSum\etaT^2\Big)^\frac12,
\end{equation}
 where $\cE_h^i$ is the set of the edges of $\cT_h$ interior to $\O$,
\begin{align}
 %\eta_p&=\hp\lambda_p,\label{eq:etapDef}\\
 \etaeO&=\frac{\sigma}{|e|^\frac12}\|\jump{u_h}\|_{L_2(e)},\label{eq:etae1Def}\\
 \etaeTwo&=|e|^\frac12\|\jumpTwo{u_h}\|_{L_2(e)},\label{eq:etae2Def}\\
 \etaeThree&=|e|^\frac32\|\jumpThree{u_h}\|_{L_2(e)},\label{eq:etae3Def}\\
 \etaT&=\hT^2\|f-\Delta^2 u_h\|_{L_2(T)}.\label{eq:etaTDef}
\end{align}
\par
 The following result will play an important role in the {\em a posteriori} error analysis
 of the obstacle problem.   Note that its proof is made simple by the representation of
 the discrete Lagrange multiplier $\lambda_h$ as a sum of Dirac point measures supported
 at the vertices of $\cT_h$, which allows the analysis in \cite{Brenner:2012:C0IP}
 to be used here.
\begin{lemma}\label{lem:ZhReliability}
  There exists a positive constant $C$, depending only on $k$ and
  the shape regularity of $\cT_h$, such that
\begin{equation}\label{eq:zhuhEstimate}
  \|z_h-u_h\|_h\leq C\eta_h.
\end{equation}
\end{lemma}
\begin{proof}  We have an obvious estimate
\begin{equation}\label{eq:zhuhEst1}
  \ESum\frac{\sigma}{|e|}\left\|\Jump{(z_h-u_h)}\right\|_{L_2(e)}^2
    =\ESum\frac{\sigma}{|e|}\left\|\Jump{u_h}\right\|_{L_2(e)}^2\leq \ESum \etaeO^2,
\end{equation}
 and it only remains to estimate $\TSum |z_h-u_h|_{H^2(T)}^2$.
\par
 Let $E_h:V_h\longrightarrow \ES$ be the enriching operator.  It follows from
 \eqref{eq:EhEnergyEst} and \eqref{eq:etae1Def} that
\begin{align}\label{eq:zhuhEst2}
  \TSum |z_h-u_h|_{H^2(T)}^2&\leq 2\TSum\big[
  |z_h-E_hu_h|_{H^2(T)}^2+|u_h-E_hu_h|_{H^2(T)}^2\big]\\
     &\leq 2|z_h-E_hu_h|_{H^2(\O)}^2+C\ESum\etaeO^2,\notag
\end{align}
 and, by duality,
\begin{equation}\label{eq:zhuhEst3}
  |z_h-E_hu_h|_{H^2(\O)}=\sup_{\phi\in\ES\setminus\{0\}}\frac{a(z_h-E_hu_h,\phi)}{|\phi|_{H^2(\O)}}.
\end{equation}
\par
 In view of \eqref{eq:lambdahDef} and \eqref{eq:zhDef},
 the numerator on the right-hand side of \eqref{eq:zhuhEst3}
 becomes
\begin{align*}%\label{eq:zhuhEst4}
  &a(z_h-E_hu_h,\phi)=\TSum \int_T D^2(z_h-E_hu_h):D^2\phi\,dx\notag\\
     &\hspace{40pt}=(f,\phi)+\sum_{p\in\cV_h}\lambda_h(p)\phi(p)\notag\\
     &\hspace{70pt}+\TSum \int_T D^2(u_h-E_hu_h):D^2\phi\,dx
     -\TSum\int_T D^2u_h:D^2(\phi-\Pi_h\phi)\,dx\\
     &\hspace{100pt}-\TSum\int_T D^2u_h:D^2(\Pi_h\phi)\,dx\\
     &\hspace{40pt}=(f,\phi)+\sum_{p\in\cV_h}\lambda_h(p)\phi(p)-(f,\Pi_h\phi)
       -\sum_{p\in\cV_h}\lambda_h(p)(\Pi_h\phi)(p)+a_h(u_h,\Pi_h\phi)\notag\\
     &\hspace{70pt}+\TSum \int_T D^2(u_h-E_hu_h):D^2\phi\,dx
     -\TSum\int_T D^2u_h:D^2(\phi-\Pi_h\phi)\,dx\\
     &\hspace{100pt}-\TSum\int_T D^2u_h:D^2(\Pi_h\phi)\,dx.
\end{align*}
 Since $\phi$ and $\Pi_h\phi$ agree on the vertices of $\cT_h$, the two terms involving
 $\lambda_h$ cancel each other and we end up with
\begin{align*}%\label{eq:zhuhEst4}
  a(z_h-E_hu_h,\phi)
  =&\TSum \int_T D^2(u_h-E_hu_h):D^2\phi\,dx
     -\TSum\int_T D^2u_h:D^2(\phi-\Pi_h\phi)\,dx\\
      &\hspace{40pt}-\TSum\int_T D^2u_h:D^2(\Pi_h\phi)\,dx+a_h(u_h,\Pi_h\phi)+(f,\phi-\Pi_h\phi),\notag
\end{align*}
 which is precisely the equation \cite[$(7.9)$]{Brenner:2012:C0IP} (and
 which has nothing to do with either
 $z_h$ or $\lambda_h$).
\par
 It then follows from the estimates \cite[$(7.10)-(7.19)$]{Brenner:2012:C0IP} that
\begin{equation}\label{eq:zhuhEst4}
  a(u-E_hu_h,\phi)\leq C\eta_h|\phi|_{H^2(\O)}.
\end{equation}
\par
 The estimate \eqref{eq:zhuhEstimate} follows from \eqref{eq:hNormDef} and
  \eqref{eq:zhuhEst1}--\eqref{eq:zhuhEst4}.
\end{proof}
%
%%%%%%%%%%%%%%%%%%%%%%%%%%%%%%%%%%
\section{Reliability Estimates for the Obstacle Problem}\label{sec:Reliability}
 We begin with a simple estimate.
\begin{lemma}\label{lem:UpperBdd}
 There exists a positive constant $C$, depending only on $k$ and
 the shape regularity of $\cT_h$,  such that
\begin{equation}\label{eq:UpperBdd}
 \|u-u_h\|_h+\|\lambda-\lambda_h\|_{H^{-2}(\O)}\leq C\eta_h+\sqrt{\int_\O (\psi-E_hu_h)^+d\lambda}\;.
\end{equation}
\end{lemma}
\begin{proof}
  Let $E_h:V_h\longrightarrow H^2_0(\O)$ be the enriching operator.
   We can write
\begin{align}\label{eq:UpperBdd1}
  |u-E_hu_h|_{H^2(\O)}^2&=a(u-E_hu_h,u-E_hu_h)\\
      &=a(u-z_h,u-E_hu_h)+a(z_h-E_hu_h,u-E_hu_h),\notag
\end{align}
 and, in view of  \eqref{eq:SameNorm},
 \eqref{eq:EhhNormEst}, \eqref{eq:etae1Def} and Lemma~\ref{lem:ZhReliability},
 the second term on the right-hand side of \eqref{eq:UpperBdd1} is bounded by
\begin{align}\label{eq:UpperBdd2}
  a(z_h-E_hu,u-E_hu_h)&\leq |z_h-E_hu_h|_{H^2(\O)}|u-E_hu_h|_{H^2(\O)}\notag\\
    &\leq \big(\|z_h-u_h\|_h+\|u_h-E_hu_h\|_h\big)|u-E_hu_h|_{H^2(\O)}\\
    &\leq C\eta_h|u-E_hu_h|_{H^2(\O)}.\notag
\end{align}
\par
 By \eqref{eq:KDef}--\eqref{eq:lambdaDef}, \eqref{eq:DiscreteComplentarityCondition},
 \eqref{eq:SignOflambdah},
  \eqref{eq:Invariance} and \eqref{eq:zhDef},
 the first term on the right-hand side of \eqref{eq:UpperBdd1} can be bounded as follows:
\begin{align}\label{eq:UpperBdd3}
  &a(u-z_h,u-E_hu_h)=\int_\O (u-E_hu_h)\,d\lambda-
     \sum_{p\in\cV_h}\lambda_h(p)\big(u(p)-(E_hu_h)(p)\big)\\
     &\hspace{40pt}=\int_\O (\psi-E_hu_h)\,d\lambda-
     \sum_{p\in\cV_h}\lambda_h(p)\big(u(p)-\psi(p)\big)\leq \int_\O (\psi-E_hu_h)^+\,d\lambda. \notag
\end{align}
\par
 It follows from \eqref{eq:SameNorm} and \eqref{eq:UpperBdd1}--\eqref{eq:UpperBdd3} that
\begin{equation*}
 \|u-E_hu_h\|_h\leq C\eta_h+\sqrt{\int_\O (\psi-E_hu_h)^+d\lambda}\;,
\end{equation*}
 which together with \eqref{eq:EhhNormEst} implies
\begin{equation}\label{eq:UpperBdd4}
 \|u-u_h\|_h\leq C\eta_h+\sqrt{\int_\O (\psi-E_hu_h)^+d\lambda}\;.
\end{equation}
\par
   In order to estimate $\|\lambda-\lambda_h\|_{H^{-2}(\O)}$, we observe that
   \eqref{eq:lambdaDef}, \eqref{eq:SameNorm} and \eqref{eq:zhDef}
   imply
\begin{align}\label{eq:LMEst}
  \|\lambda-\lambda_h\|_{H^{-2}(\O)}&=\sup_{v\in H^2_0(\O)}\frac{\d\int_\O v\,d(\lambda-\lambda_h)}
   {|v|_{H^2(\O)}}\\
   &=\sup_{v\in H^2_0(\O)}\frac{a(u-z_h,v)}
   {|v|_{H^2(\O)}}
   = |u-z_h|_{H^2(\O)}\
  \leq \|u-u_h\|_h+\|z_h-u_h\|_h.\notag
\end{align}
 The estimate  for $\|\lambda-\lambda_h\|_{H^{-2}(\O)}$ then follows from Lemma~\ref{lem:ZhReliability}
  and \eqref{eq:UpperBdd4}.
\end{proof}
\par
 We can also remove the inconvenient $E_h$ in the estimate \eqref{eq:UpperBdd}.
\begin{theorem}\label{thm:ReliableII}
 There exists a positive constant $C$, depending only on $k$ and
 the shape regularity of $\cT_h$, such that
\begin{align}\label{eq:ReliableII}
\|u-u_h\|_h+\|\lambda-\lambda_h\|_{H^{-2}(\O)}
 &\leq C\Big(\eta_h+|\lambda|^\frac12\sqrt{\max_{T\in\cT_h}\hT\sum_{ e\in\tilde\cE_T}|e|^{-1/2}\|\jump{u_h}\|_{L_2(e)}}\,\Big)\\
   &\hspace{40pt}+ |\lambda|^\frac12\|(\psi-u_h)^+\|_{L_\infty(\O)}^\frac12,\notag
\end{align}
 where $\tilde\cE_T$ is the set of the edges in $\cT_h$ that emanate from the vertices of $T$.
\end{theorem}
\begin{proof}  We have
\begin{equation}\label{eq:uReliable1}
 \int_\O  (\psi-E_hu_h)^+\,d\lambda \leq
  \big[\|(\psi-u_h)^+\|_{L_\infty(\O)}+\|u_h-E_hu_h\|_{L_\infty(\O)}\big]|\lambda|,
\end{equation}
 and, by \eqref{eq:EhLInfty},
\begin{equation}\label{eq:uReliable2}
  \|u_h-E_hu_h\|_{L_\infty(\O)}\leq C\max_{T\in\cT_h}\hT\sum_{ e\in\tilde\cE_T}|e|^{-1/2}\|\jump{u_h}\|_{L_2(e)}.
\end{equation}
\par
 The estimate \eqref{eq:ReliableII} follows from  \eqref{eq:UpperBdd},  \eqref{eq:uReliable1}, and \eqref{eq:uReliable2}.
\end{proof}
\begin{remark}\label{rem:ReliableII}
 The estimate \eqref{eq:ReliableII} is {not} a genuine {\em a posteriori} error estimate since $|\lambda|$ is not known.
 But it is useful for monitoring the asymptotic convergence of adaptive algorithms
 (cf. Lemma~\ref{lem:AsymptoticConvergenceRate} and Lemma~\ref{lem:LMTest}).
\end{remark}
\begin{remark}\label{rem:Genuine}
  Under the stronger assumption $\psi\in C^2(\bar\O)$ on the obstacle function,
  one can also obtain a genuine
 {\em a posteriori} error estimate by replacing $|\lambda|$ with a computable bound.
\par
  Indeed, for any $w\in K$, we have
\begin{equation*}
  \frac12|u|_{H^2(\O)}^2\leq\frac12 |w|_{H^2(\O)}^2-(f,w)+(f,u)
  \leq \frac12 |w|_{H^2(\O)}^2-(f,w)+C\|f\|_{L_2(\O)}^2+\frac14|u|_{H^2(\O)}^2,
\end{equation*}
 by a Poincar\'e-Friedrichs inequality \cite{Necas:2012:Direct} and the arithmetic-geometric means inequality,
 and hence
\begin{equation}\label{eq:uH2Bdd}
  |u|_{H^2(\O)}^2\leq 2|w|_{H^2(\O)}^2-4(f,w)+C\|f\|_{L_2(\O)}^2,
\end{equation}
 where $C$ is a computable positive constant.  Combining \eqref{eq:uH2Bdd} with the Sobolev
 embedding (cf. \cite{ADAMS:2003:Sobolev})
   $H^2(\O)\hookrightarrow C^{0,\gamma}(\O)$
 for any $\gamma<1$, we see that there is a computable $\delta>0$ such that
 $u(x)>\psi(x)$ if the distance from $x$ to $\p\O$ is $<\delta$.
 Therefore there is a computable $\phi\in C^\infty_c(\O)$ such that
 $\phi=1$ on the support of $\lambda$.
\par
 We then have, in view of \eqref{eq:lambdaDef} and \eqref{eq:uH2Bdd},
\begin{equation*}%\label{eq:LambdaEst}
  |\lambda|=a(u,\phi)-(f,\phi)\leq |u|_{H^2(\O)}|\phi|_{H^2(\O)}+\|f\|_{L_2(\O)}\|\phi\|_{L_2(\O)}\leq C,
 \end{equation*}
 where the positive constant $C$ is computable.
\end{remark}
%
%%%%%%%%%%%%%%%%%%%%%%%%%%%%%%%%%%%%%%%%%%%%%%%%%%%%%%%%%%%
\section{Efficiency Estimates for the Obstacle Problem}\label{sec:Efficiency}
\par
 Let the local data oscillation $\Osc(f;T)$ be defined by
\begin{equation*}
 \Osc(f;T)=\hT^2\|f-\bfT\|_{L_2(T)},
\end{equation*}
 where $\bfT$ is the $L_2$ projection of $f$ in the polynomial space $P_j(T)$ with
 $j=\max(k-4,0)$.  The global data oscillation is then given by
 $$\OSC=\Big(\sum_{T\in\cT_h}\Osc(f;T)^2\Big)^\frac12.$$
\begin{theorem}\label{thm:LocalEfficiencyII}
  There exists a positive constant $C$, depending only on the shape regularity of $\cT_h$, such that
\begin{alignat*}{3}
  \etaeO&\leq \frac{\sigma}{|e|^\frac12} \|\jump{(u-u_h)}\|_{L_2(e)}&\quad&\forall\,e\in\cE_h,\\
  \etaeTwo&\leq C\Big[\sum_{T\in\cT_e}\big[|u-u_h|_{H^2(T)}+\Osc(f;T)\big]+\|\lambda-\lambda_h\|_{H^{-2}(\Omega_e)}
 \Big]&\qquad&\forall\,e\in\cE_h^i,\\
  \etaeThree&\leq C\Big[\sum_{T\in\cT_e}\big[|u-u_h|_{H^2(T)}+\Osc(f;T)\big]+\|\lambda-\lambda_h\|_{H^{-2}(\Omega_e)}\\
 &\hspace{40pt}+\frac{1}{|e|}\|\jump{(u-u_h)}\|_{L_2(e)}^2\Big]&\qquad&\forall\,e\in\cE_h^i,\\
  \etaT&\leq C\big(|u-u_h|_{H^2(T)}+\Osc(f;T)+\|\lambda-\lambda_h\|_{H^{-2}(T)}
   \big)&\qquad&\forall\,T\in\cT_h,
\end{alignat*}
 where $\cT_e$ is the set of the two triangles that share the edge $e$ and
 $\O_e$ is the interior of $\bigcup_{T\in\cT_e}\bar T$.
\end{theorem}
\begin{proof}
  The estimate for $\eta_{e,1}$ is obvious.  The other estimates are obtained by modifying the
  arguments in \cite[Section~5.3]{Brenner:2012:C0IP}.
\par
 In the proof of the estimate \cite[$(5.17)$]{Brenner:2012:C0IP} (with $v=u_h$), we replace the relation
\begin{equation*}
  \int_T (\bfT-\Delta^2 u_h)z\,dx=\int_T D^2(u-u_h):D^2z\,dx+\int_T (\bfT-f)z\,dx
\end{equation*}
 by
\begin{align}\label{eq:AreaBubble}
  \int_T (\bfT-\Delta^2 u_h)z\,dx=\int_T D^2(u-u_h):D^2z\,dx+\int_T (\bfT-f)z\,dx
 -\int_T z\,d(\lambda-\lambda_h)
\end{align}
 to obtain the estimate
\begin{align*}
  \int_T (\bfT-\Delta^2u_h)z&\leq C\big(\hT^{-2}|u-u_h|_{H^2(T)}
  +\|f-\bfT\|_{L_2(T)}+\hT^{-2}\|\lambda-\lambda_h\|_{H^{-2}(T)}\big)\|z\|_{L_2(T)},
\end{align*}
 which then leads to the estimate for $\etaT$.  Note that \eqref{eq:AreaBubble} holds because
 the bubble function $z$ vanishes at the vertices of $\cT_h$.
\par
 In the proof of the estimate  \cite[$(5.26)$]{Brenner:2012:C0IP} (with $v=u_h$), we replace the relation
\begin{align*}
 &\sum_{T\in\cT_e}\Big(-\int_T D^2u_h:D^2(\zeta_1\zeta_2)\,dx
 +\int_T(\Delta^2u_h)(\zeta_1\zeta_2)\,dx\Big)\\
 &\hspace{40pt}=
    \sum_{T\in\cT_e}\int_TD^2(u-u_h):D^2(\zeta_1\zeta_2)\,dx-\sum_{T\in\cT_e}\int_T (f-\Delta^2u_h)
    (\zeta_1\zeta_2)\,dx
\end{align*}
 that appears in \cite[$(5.24)$]{Brenner:2012:C0IP} by
\begin{align}\label{eq:EdgeBubble1}
  &\sum_{T\in\cT_e}\Big(-\int_T D^2u_h:D^2(\zeta_1\zeta_2)\,dx
 +\int_T(\Delta^2u_h)(\zeta_1\zeta_2)\,dx\Big)\notag\\
  &\hspace{40pt}=
  \sum_{T\in\cT_e}\int_TD^2(u-u_h):D^2(\zeta_1\zeta_2)\,dx-\sum_{T\in\cT_e}\int_T
  (f-\Delta^2u_h)(\zeta_1\zeta_2)\,dx\\
    &\hspace{70pt} -\int_{\O_e}(\zeta_1\zeta_2)\,d(\lambda-\lambda_h)\notag
\end{align}
 to obtain the estimate
\begin{align*}
  &\sum_{T\in\cT_e}\Big(-\int_T D^2u_h:D^2(\zeta_1\zeta_2)\,dx
 +\int_T(\Delta^2u_h)(\zeta_1\zeta_2)\,dx\Big)\\
     &\hspace{30pt}\leq C\Big[\sum_{T\in\cT_e}\big(\hT^{-2}|u-u_h|_{H^2(T)}+\|f-\Delta^2u_h\|_{L_2(T)}\big)
     +\hT^{-2}\|\lambda-\lambda_h\|_{H^{-2}(\O_e)}\Big]\|\zeta_1\zeta_2\|_{L_2(\O_e)},
\end{align*}
 which then leads to the estimate for $\eta_{e,2}$.  Note that
 \eqref{eq:EdgeBubble1} holds because the bubble function $\zeta_1\zeta_2$ vanishes
 at the vertices of $\cT_h$.
\par
 Finally, in the proof of the estimate \cite[$(5.32)$]{Brenner:2012:C0IP} (with $v=u_h$),
 we replace the relation
\begin{align*}
  &\sum_{T\in\cT_e}\Big(\int_T D^2u_h:D^2(\zeta_2\zeta_3)\,dx
    -\int_T (\Delta^2 u_h)(\zeta_2\zeta_3)\,dx\Big)\\
    &\hspace{40pt}=\sum_{T\in\cT_e}\int_T D^2(u_h-u):D^2(\zeta_2\zeta_3)\,dx+
        \sum_{T\in\cT_2}\int_T(f-\Delta^2u_h)(\zeta_2\zeta_3)\,dx
\end{align*}
 that appears in \cite[$(5.30)$]{Brenner:2012:C0IP} by
\begin{align}\label{eq:EdgeBubble2}
  &\sum_{T\in\cT_e}\Big(\int_T D^2u_h:D^2(\zeta_2\zeta_3)\,dx
    -\int_T (\Delta^2 u_h)(\zeta_2\zeta_3)\,dx\Big)\notag\\
    &\hspace{40pt}=\sum_{T\in\cT_e}\int_T D^2(u_h-u):D^2(\zeta_2\zeta_3)\,dx+
        \sum_{T\in\cT_2}\int_T(f-\Delta^2u_h)(\zeta_2\zeta_3)\,dx\\
        &\hspace{70pt}+\int_{\O_e} (\zeta_2\zeta_3)\,d(\lambda-\lambda_h)
        \notag
\end{align}
 to obtain the estimate
\begin{align*}
  &\sum_{T\in\cT_e}\Big(\int_T D^2u_h:D^2(\zeta_2\zeta_3)\,dx
    -\int_T (\Delta^2 u_h)(\zeta_2\zeta_3)\,dx\Big)\\
     &\hspace{30pt}\leq C\Big[\sum_{T\in\cT_e}\big(\hT^{-2}|u-u_h|_{H^2(T)}+\|f-\Delta^2u_h\|_{L_2(T)}\big)
     +\hT^{-2}\|\lambda-\lambda_h\|_{H^{-2}(\O_e)}\Big]\|\zeta_2\zeta_3\|_{L_2(\O_e)},
\end{align*}
 which then leads to the estimate for $\etaeThree$.  Again \eqref{eq:EdgeBubble2} holds because
 the bubble function $\zeta_2\zeta_3$ vanishes at the vertices of $\cT_h$.
\end{proof}
\par
 We can also prove a global efficiency result under the following assumption:
\begin{align}\label{eq:RefinementLevels}
 &\text{The triangles (resp. interior edges) of $\cT_h$ can be divided into  $n$ disjoint
  groups}\notag\\
  &\text{so that the ratio of the diameters of any two triangles
   (resp. interior edges) in}\\
   &\text{the same group is bounded above by a constant $\tau\geq1$.}\notag
\end{align}
\begin{theorem}\label{thm:GlobalEfficiency}
  Under assumption \eqref{eq:RefinementLevels},
  there exists a positive constant $C$ depending only on $\tau$, $k$ and
   the shape regularity of $\cT_h$ such that
\begin{equation}\label{eq:GlobalEfficiency}
  \eta_h\leq C\big(\sqrt{\sigma}\|u-u_h\|_h+\sqrt{n}\|\lambda-\lambda_h\|_{H^{-2}(\O)}+\OSC\big).
\end{equation}
\end{theorem}
\begin{proof}  We have a trivial estimate
 \begin{equation*}
   \ESum\etaeO^2\leq C\ESum\frac{\sigma^2}{|e|}\left\|\Jump{(u-u_h)}\right\|_{L_2(e)}^2.
 \end{equation*}
\par
 For the estimate involving $\etaT$,
 we first write
  $\cT_h$ as the disjoint union $\cT_{h,1}\cup\cdots\cup\cT_{h,n}$ so that the  ratio of the diameters of
  any two triangles in $\cT_{h,j}$
 is bounded by $\tau$.  For $1\leq j\leq n$, the subdomain $\O_j$ is the interior of
 $\cup_{T\in\cT_{h,j}}\bar T$.
\par
 For any $T\in\cT_{h,j}$, let $\zT$ be the bubble function  in \cite[Section~5.3.2]{Brenner:2012:C0IP}
 associated with $T$ and we define
 $z_j=\sum_{T\in\cT_{h,j}}\zT \in H^2_0(\O_j)$.
   It follows from \cite[$(5.16)$]{Brenner:2012:C0IP}, \eqref{eq:AreaBubble}
   and a standard inverse estimate that
\begin{align*}
 \|\bfT-\Delta^2u_h\|_{L_2(T)}^2&\leq C\int_T (\bfT-\Delta^2u_h)\zT\,dx\\
 &\leq C\Big(\big[\hT^{-2}|u-u_h|_{H^2(T)}
  +\|f-\bfT\|_{L_2(T)}\big]\|\zT\|_{L_2(T)}-\int_T \zT\,d(\lambda-\lambda_h)\Big)
\end{align*}
 and hence
\begin{align*}
 &\sum_{T\in\cT_{h,j}}\|\bfT-\Delta^2 u_h\|_{L_2(T)}^2
  \leq C\Big(\sum_{T\in\cT_{h,j}}\big[\hT^{-2}|u-u_h|_{H^2(T)}
  +\|f-\bfT\|_{L_2(T)}\big]\|\zT\|_{L_2(T)}\\
     &\hspace{180pt}-\int_{\O_j}z_j\,d(\lambda-\lambda_h)\Big)\\
     &\hspace{40pt}\leq C\Big[\Big(\sum_{T\in\cT_{h,j}}\big[\hT^{-4}|u-u_h|_{H^2(T)}^2
     +\|f-\bfT\|_{L_2(T)}^2\big]\Big)^\frac12\Big(\sum_{T\in\cT_{h,j}}\|\zT\|_{L_2(T)}^2\Big)^\frac12\\
     &\hspace{80pt}+
     \|\lambda-\lambda_h\|_{H^{-2}(\O_j)}
     \Big(\sum_{T\in\cT_{h,j}}\hT^{-4}\|\zT\|_{L_2(T)}^2\Big)^\frac12\Big]
\end{align*}
 by a standard inverse estimate.
\par
 Therefore we have
\begin{align}\label{eq:BubbleEst1}
 \sum_{T\in\cT_{h,j}}\hT^4\|\bar f-\Delta^2u_h\|_{L_2(T)}^2 &\leq
 C\Big(\sum_{T\in\cT_{h,j}}\big[\hT^4\|f-\bar f\|_{L_2(T)}^2+
    |u-u_h|_{H^2(T)}^2\big] \\
    &\hspace{60pt}+\|\lambda-\lambda_h\|_{H^{-2}(\O_j)}^2\Big)\notag
\end{align}
because (cf. \cite[$(5.16)$]{Brenner:2012:C0IP})
  $$\|\zT\|_{L_2(T)}\approx \|\bar f-\Delta^2u_h\|_{L_2(T)}$$
 and the diameters $\hT$ are comparable for $T\in\cT_{h,j}$.
\par
 It follows from \eqref{eq:BubbleEst1} that
\begin{align*}
 \TSum\etaT^2&= \TSum\hT^4\|f-\Delta^2u_h\|_{L_2(T)}^2\\
    &\leq
    \sum_{j=1}^n \sum_{T\in\cT_{h,j}}2\hT^4\big[\|\bfT-\Delta^2u_h\|_{L_2(T)
    }+\|f-\bfT\|_{L_2(T)}\big]^2\\
    &\leq  C\sum_{j=1}^n \Big(\sum_{T\in\cT_{h,j}}\big[\hT^4\|f-\bfT\|_{L_2(T)}^2+
    |u-u_h|_{H^2(T)}^2\big] \\
    &\hspace{60pt}+\|\lambda-\lambda_h\|_{H^{-2}(\O_j)}^2\Big)\\
    &\leq C\Big(\OSC^2+\sum_{T\in\cT_h}|u-u_h|_{H^2(T)}^2
    +n\|\lambda-\lambda_h\|_{H^{-2}(\O)}^2\Big),
\end{align*}
 where we have also used the trivial estimate
$
  \|\lambda-\lambda_h\|_{H^{-2}(\O_j)}\leq \|\lambda-\lambda_h\|_{H^{-2}(\O)}.
$
\par
 The estimates for $\etaeTwo$ and $\etaeThree$ can be established by
 using \eqref{eq:EdgeBubble1}, \eqref{eq:EdgeBubble2} and results in
 \cite[Sections~5.3.3 and 5.3.4]{Brenner:2012:C0IP}.  Their derivations are
 similar to the derivation for $\etaT$ and hence are omitted.
\end{proof}
%
%%%%%%%%%%%%%%%%%%%%%%%%%%%%%%%%%%%%%%%%%%%
\section{An Adaptive Algorithm}\label{sec:Adaptive}
 In view of the efficiency estimates in Section~\ref{sec:Efficiency},
 we will use $\eta_h$ from \eqref{eq:BVPEstimator} as the error indicator in the adaptive loop
\begin{equation*}
  \textsf{Solve}\longrightarrow \textsf{Estimate}\longrightarrow \textsf{Mark}\longrightarrow
  \textsf{Refine}
\end{equation*}
 to define an adaptive algorithm for the $C^0$ interior penalty methods for
 \eqref{eq:Obstacle}--\eqref{eq:KDef}.
\par
 In the step \textsf{Solve}, we compute the solution of the discrete obstacle problem \eqref{eq:C0IP} by
 a primal-dual active set method \cite{BIK:1999:PDAS,HIK:2003:PDAS}.  In the step
 \textsf{Estimate}, we
 compute  $\etaeO$, $\etaeTwo$, $\etaeThree$
 and $\etaT$ defined in \eqref{eq:etae1Def}--\eqref{eq:etaTDef}.  In the step \textsf{Mark}, we use
 the D\"orfler marking strategy \cite{Dorfler:1996:AdaptiveConvergence} to mark a minimum number of
 triangles and edges whose contributions exceed $\theta\eta_h$ for some
 $\theta\in(0,1)$.  In the step \textsf{Refine},
 we refine the marked triangles and edges followed by a closure algorithm that preserves the conformity
 of the triangulation.
\par
  In the adaptive setting the subscript $h$ will be replaced by the subscript $\ell$,  where $\ell=0,1,\,\ldots$ denotes the
 level of refinements.
 The adaptive algorithm generates a sequence of triangulations $\cT_\ell$ of $\O$,
 a sequence of  solutions
 $u_\ell\in V_\ell$ of the discrete obstacle problems and a sequence of error indicators
 $\eta_\ell$.
\par
 According to Theorem~\ref{thm:ReliableII}, we can use the following result to monitor
 the asymptotic convergence rate of the adaptive algorithm.
\begin{lemma}\label{lem:AsymptoticConvergenceRate}
 Suppose  $\eta_\ell=O(N_\ell^{-\gamma})$, where $N_\ell$ is the number of degrees of freedom
 $($dof$)$  at
 the refinement level $\ell$.  Then  we have
 \begin{equation}\label{eq:AsymptoticConvergence}
  \|u-u_\ell\|_\ell+\|\lambda-\lambda_\ell\|_{H^{-2}(\O)}=O(N_\ell^{-\gamma})
 \end{equation}
  provided that
\begin{align}
    Q_{\ell,1}=\sqrt{\max_{T\in\cT_\ell}\hT\sum_{ e\in\tilde\cE_T}|e|^{-1/2}\|\jump{u_\ell}\|_{L_2(e)}}&=O(N_\ell^{-\gamma}),\label{eq:MaxJump}\\
   Q_{\ell,2}=\|(\psi-u_\ell)^+\|_{L_\infty(\O)}^\frac12 &=O(N_\ell^{-\gamma}).\label{eq:Violation}
\end{align}
 In particular,  the estimate \eqref{eq:AsymptoticConvergence} holds if $Q_{\ell,1}$
  and $Q_{\ell,2}$ are dominated  by $\eta_\ell$.
\end{lemma}
\par\smallskip
 Note that $\|\lambda-\lambda_h\|_{H^{-2}(\O)}$ is not computable.  However we can
 test the convergence of $\|\lambda-\lambda_h\|_{H^{-2}(\O)}$ indirectly as follows.
 Let $\phi\in C^\infty_c(\O)$ be equal to $1$ on the supports of $\lambda$ and the $\lambda_\ell$'s.
 Then we have
\begin{equation}\label{eq:LMIndirect}
  |\lambda|-|\lambda_\ell|=\int_\O \phi\,d(\lambda-\lambda_\ell)
      \leq |\phi|_{H^2(\O)}\|\lambda-\lambda_\ell\|_{H^{-2}(\O)},
\end{equation}
 which implies
\begin{align}\label{eq:DscreteLMEstimates} |\lambda_\ell|-|\lambda_{\ell+1}|
&=(|\lambda_\ell|-|\lambda|)+(|\lambda|-|\lambda_{\ell+1}|)\\
        &\leq |\phi|_{H^2(\O)}\big(\|\lambda-\lambda_\ell\|_{H^{-2}(\O)}
         +\|\lambda-\lambda_{\ell+1}\|_{H^{-2}(\O)}\big).\notag
\end{align}
\par
 Let $\Lambda_\ell$ be defined by
\begin{equation}\label{eq:LambdaEllDef}
  \Lambda_\ell=|(|\lambda_\ell|-|\lambda_{\ell+1}|)|.
\end{equation}
 The following result is an immediate consequence of Lemma~\ref{lem:AsymptoticConvergenceRate}
  and \eqref{eq:DscreteLMEstimates}.
\begin{lemma}\label{lem:LMTest}
    Suppose  $\eta_\ell=O(N_\ell^{-\gamma})$, where $N_\ell$ is the number of dof  at
 the refinement level $\ell$.  Then we have
\begin{equation*}
   \Lambda_\ell=O(N_\ell^{-\gamma})
\end{equation*}
 provided that  \eqref{eq:MaxJump} and \eqref{eq:Violation} are valid.
\end{lemma}
\begin{remark}\label{rem:lambdaEst}
  In view of \eqref{eq:LMIndirect}, we can also replace $|\lambda|$ by $|\lambda_\ell|$
  in \eqref{eq:ReliableII}  to obtain a true {\em a posteriori} error estimate that is
  asymptotically reliable  under the assumptions of Lemma~\ref{lem:AsymptoticConvergenceRate}.
\end{remark}
%%%%%%%%%%%%%%%%%%%%%%%%%%%%%%%%%%%%%%%%%%%
\section{Numerical Experiments}\label{sec:Numerics}
 In this section we report numerical results that demonstrate the estimate \eqref{eq:ReliableII}
 and illustrate the performance of the adaptive algorithm %from Section~\ref{sec:Adaptive}
 for quadratic and cubic $C^0$ interior penalty methods.  We choose the penalty parameter $\sigma$
 to be
 $6$ (resp. $18$) for the quadratic (resp. cubic) $C^0$ interior penalty method.
 We also take $\theta$ to be $0.5$ in the
 D\"orfler marking strategy.
 \par
  We will consider three examples.  The first one concerns
   a problem on the unit square with known
  exact solution.  The second one is about a problem on a $L$-shaped domain with a two dimensional
  coincidence set (where $u=\psi$)
  that has a fairly smooth boundary.  The third example is also about
   a problem on a $L$-shaped domain
  but with a coincidence set that is one dimensional.   For the second and third examples
   where the exact solution is not known, we estimate the error
 $\|u-u_\ell\|_\ell$ by using a reference solution computed on the mesh obtained by
 a uniform refinement of the last mesh generated by the refinement procedure.
 \par
  In each of the experiment for the adaptive algorithm, we will present figures that display
   the convergence histories for
  $\|u-u_\ell\|_\ell$ and $\eta_\ell$, and for the quantities $Q_{\ell,1}$ and $Q_{\ell,2}$ defined in
  \eqref{eq:MaxJump} and \eqref{eq:Violation}.  We also present tables that contain
  numerical results for the quantity  $\Lambda_\ell$ defined in \eqref{eq:LambdaEllDef}
  and examples of adaptively generated meshes.
\subsection{Example 1} \label{subsec:Example1}
 In this example we consider an obstacle problem on the
 unit square $\O=(-0.5,0.5)^2$ from \cite[Example~1]{BSZZ:2012:Kirchhoff} with $f=0$, $\psi=1-|x|^2$ and nonhomogeneous
 boundary conditions,
 whose exact solution is given by
\begin{equation*}
	u(x)=\begin{cases}
	C_1|x|^2 \ln(|x|) + C_2|x|^2 + C_3\ln(|x|)+C_4 & \qquad r_0 < |x|\\[6pt]
	1-|x|^2 & \qquad |x|\leq r_0
\end{cases},
\end{equation*}
 where
   $r_0\approx 0.18134453$,
   $C_1\approx 0.52504063$,
   $C_2\approx -0.62860905$,
   $C_3\approx 0.017266401$ and
   $C_4\approx 1.0467463$.
\par
 For this example the coincidence set
  is the disc centered at the origin with radius $r_0$ whose
 boundary is the free boundary, and we have
 $|\lambda|=8\pi C_1\approx 13.1957.$
\par
  Due to the nonhomogeneous boundary conditions,  we modify
  the discrete obstacle problem (cf. \cite{BSZZ:2012:Kirchhoff}) to find
\begin{equation*}
   u_h = \argmin_{v \in K_h}\Big[\frac{1}{2} a_h(v,v) - F(v)\Big],
\end{equation*}
 where $K_h= \{ v \in V_h : v-\Pi_h u \in H^1_0(\Omega), \, v(p) \geq \psi(p)
  \quad \forall\,p \in \cV_h \}$,
\begin{align*}
 F(v) &= (f,v) + \sum_{e \in \cE_h^b} \int_e
  \left( \Mean{v} + \frac{\sigma}{|e|}
  \Jump{v} \right)
  \Jump{u}ds,
\end{align*}
 and $\cE_h^b$ is the set of the edges of $\cT_h$ that are on the boundary of $\O$.
%\par
 We also modify
 the residual based error estimator:
\begin{equation*}
  \eta_h=\Big(\ESumI\etaeO^2+\ESumI (\etaeTwo^2+\etaeThree^2)+\TSum\etaT^2
        +\sum_{e\in\cE_h^b}\sigma^2|e|^{-1}\|\jump{(u_h-u)}\|_{L_2(e)}^2\Big)^\frac12.
\end{equation*}
\par
 In the first experiment we
  solve the discrete problem with the $P_2$ element
   on uniform meshes and
 compute the quantity
\begin{align}\label{eq:QDef}
  Q_h&=C\Big(\eta_h+|\lambda|^\frac12\sqrt{\max_{T\in\cT_h}\hT\sum_{ e\in\tilde\cE_T}|e|^{-1/2}\|\jump{u_h}\|_{L_2(e)}}\,\Big)\\
   &\hspace{40pt}+ |\lambda|^\frac12\|(\psi-u_h)^+\|_{L_\infty(\O)}^\frac12\notag
\end{align}
 that appears on the right-hand side of \eqref{eq:ReliableII}, with $C=0.32$ and
 $|\lambda|=13.196$.  The results for $\|u-u_h\|_h/Q_h$ (cf.
   Table~\ref{table:ReliabilityTest})
 clearly demonstrate the estimate \eqref{eq:ReliableII}.
\begin{table}[hh]
  \begin{tabular}{c|c|c|c|c|c|c|c|c|c|c}
    $h$ & $2^{-1}$ & $2^{-2}$ &$2^{-3}$ & $2^{-4}$ &$2^{-5}$ &$2^{-6}$ &$2^{-7}$&
    $2^{-8}$ &$2^{-9}$ & $2^{-10}$\\
    \hline
    $\|u-u_h\|_h/Q_h$ & $0.93$ & $0.94$ & $0.96$ & $0.97$ & $0.97$ & $0.97$ &$0.98$&
    $0.98$ & $0.99$ & $1.04$
  \end{tabular}
  \par\medskip
   \caption{Numerical results for the estimate \eqref{eq:ReliableII}}
\label{table:ReliabilityTest}
\end{table}
%
%%%%%%%%%%%%%%%%%%%%%%%
%
\par
 In the second experiment we solve the discrete
 obstacle problem with the cubic element
 on uniform and adaptive meshes.   We observe optimal (resp. suboptimal) convergence
  rate  for adaptive (resp. uniform) meshes in Figure~\ref{fig:P3Example1}\hspace{1pt}(a)  and
  also the reliability of $\eta_\ell$.   Furthermore the optimal $O(N_\ell^{-1})$ convergence
 rate of $\|u-u_\ell\|_\ell$  is justified by
 Figure~\ref{fig:P3Example1}\hspace{1pt}(b) and Lemma~\ref{lem:AsymptoticConvergenceRate}.
\begin{figure}[hh]
  \subfigure[]
  {\includegraphics[width=0.4\linewidth]{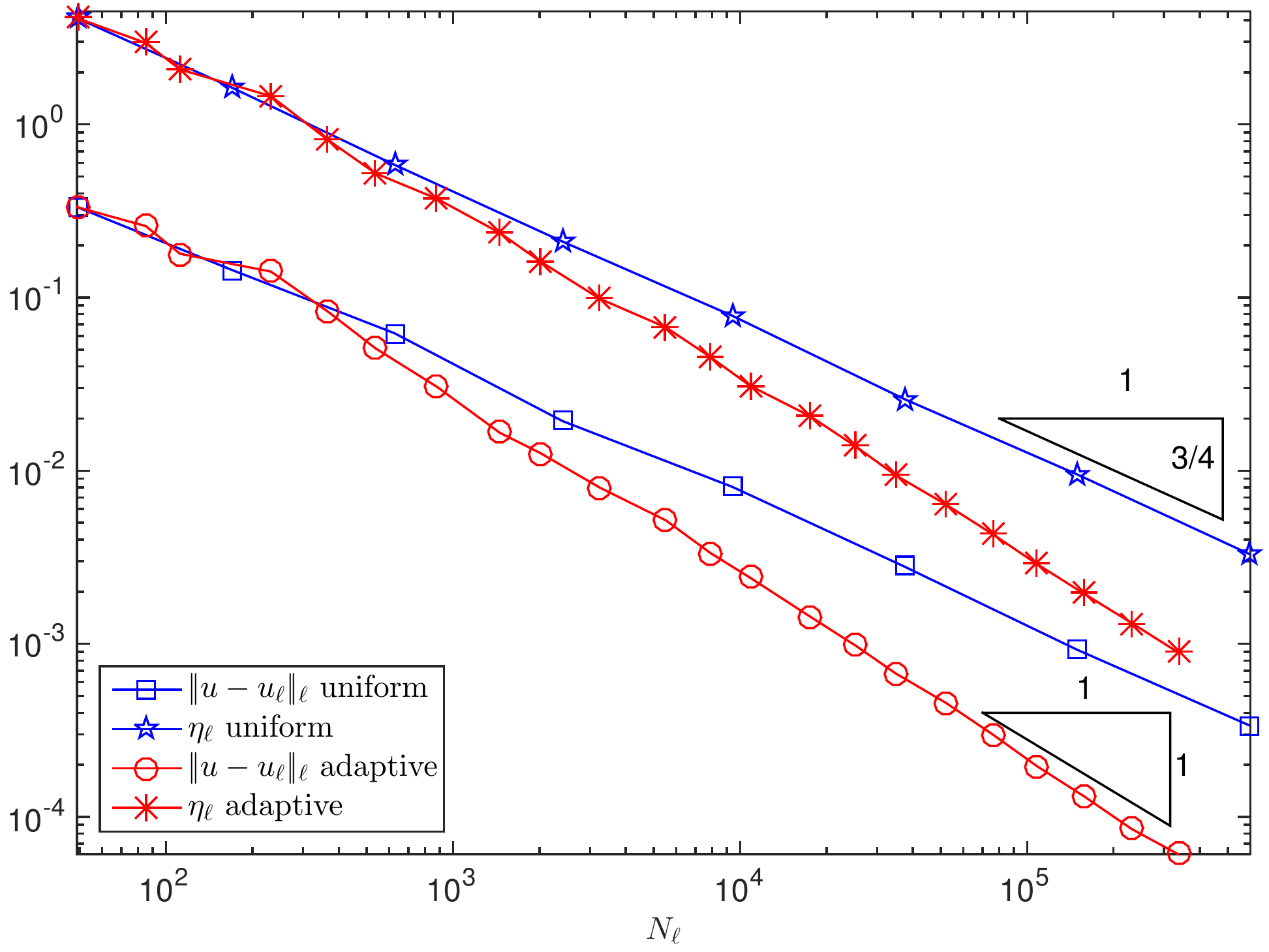}}
  \subfigure[]
  {\includegraphics[width=0.4\linewidth]{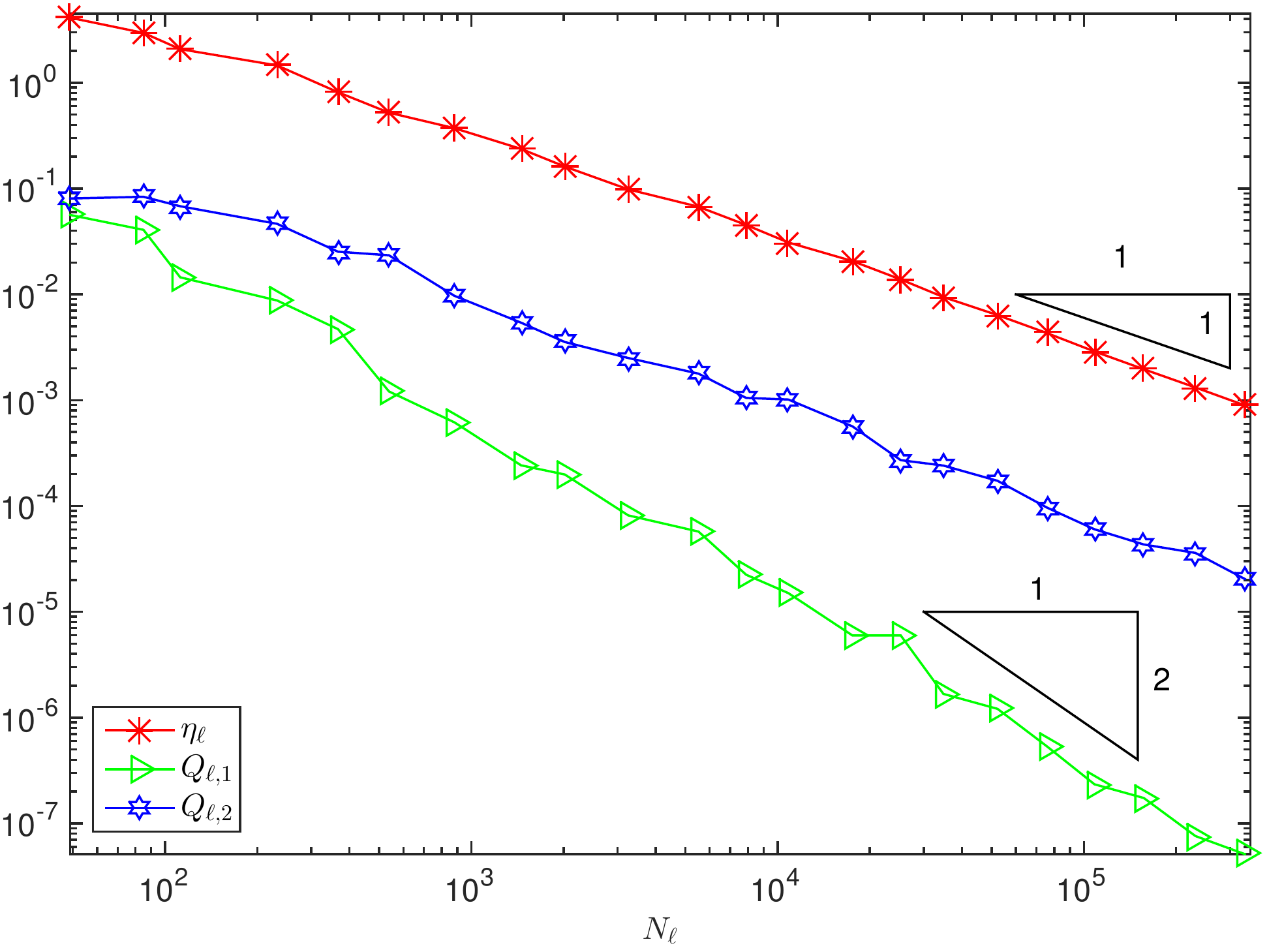}}
\caption{Convergence histories for the cubic $C^0$ interior penalty method
  for Example~1: (a) $\|u-u_\ell\|_\ell$ and
  $\eta_\ell$, (b)  $\eta_\ell$, $Q_{\ell,1}$ and $Q_{\ell,2}$}
\label{fig:P3Example1}
\end{figure}
\par
 According to Lemma~\ref{lem:LMTest} and
  Figure~\ref{fig:P3Example1}\hspace{1pt}(b), the magnitude of $\Lambda_\ell$
  should be $O(N_\ell^{-1})$.  This is confirmed by the
  results in Table~\ref{table:LMExample1}, where $N_\ell$ increases from $N_0=49$ to
  $N_{20}=231328$.
\begin{table}[hh]
\begin{tabular}{c|c|c|c|c|c|c|c|c|c|c|c}
  $\ell$ & $0$ & $1$ &$2$ & $3$ & $4$ & $5$ &$6$ & $7$ & $8$ &$9$ & $10$\\
  \hline
  $\Lambda_\ell N_\ell$ & $155$ & $250$ & $60.5$ & $214$ & $164$ & $99.6$ & $101$
  & $75.8$ & $21.8$ & $31.9$ & $17.9$ \\
  \hline  \hline
   $\ell$ & $11$ & $12$ &$13$ & $14$ & $15$ & $16$ & $17$ & $18$& $19$ &$20$ \\
   \hline
   $\Lambda_\ell N_\ell$ & $12.2$ & $1.48$ & $23.8$ & $5.37$ &$0.403$ & $4.65$
   & $20.5$ &$51.1$ & $259$ & $730$
\end{tabular}
\medskip
\caption{$\Lambda_\ell N_\ell$ for the adaptive cubic $C^0$ interior penalty method for
  Example~1}
 \label{table:LMExample1}
\end{table}
\par
 An adaptive mesh with roughly 3000 nodes is depicted in Figure~\ref{fig:MeshExp2} and
 strong refinement near the free boundary is observed.
\begin{figure}[hh]
\begin{center}
\includegraphics[width=0.32\linewidth]{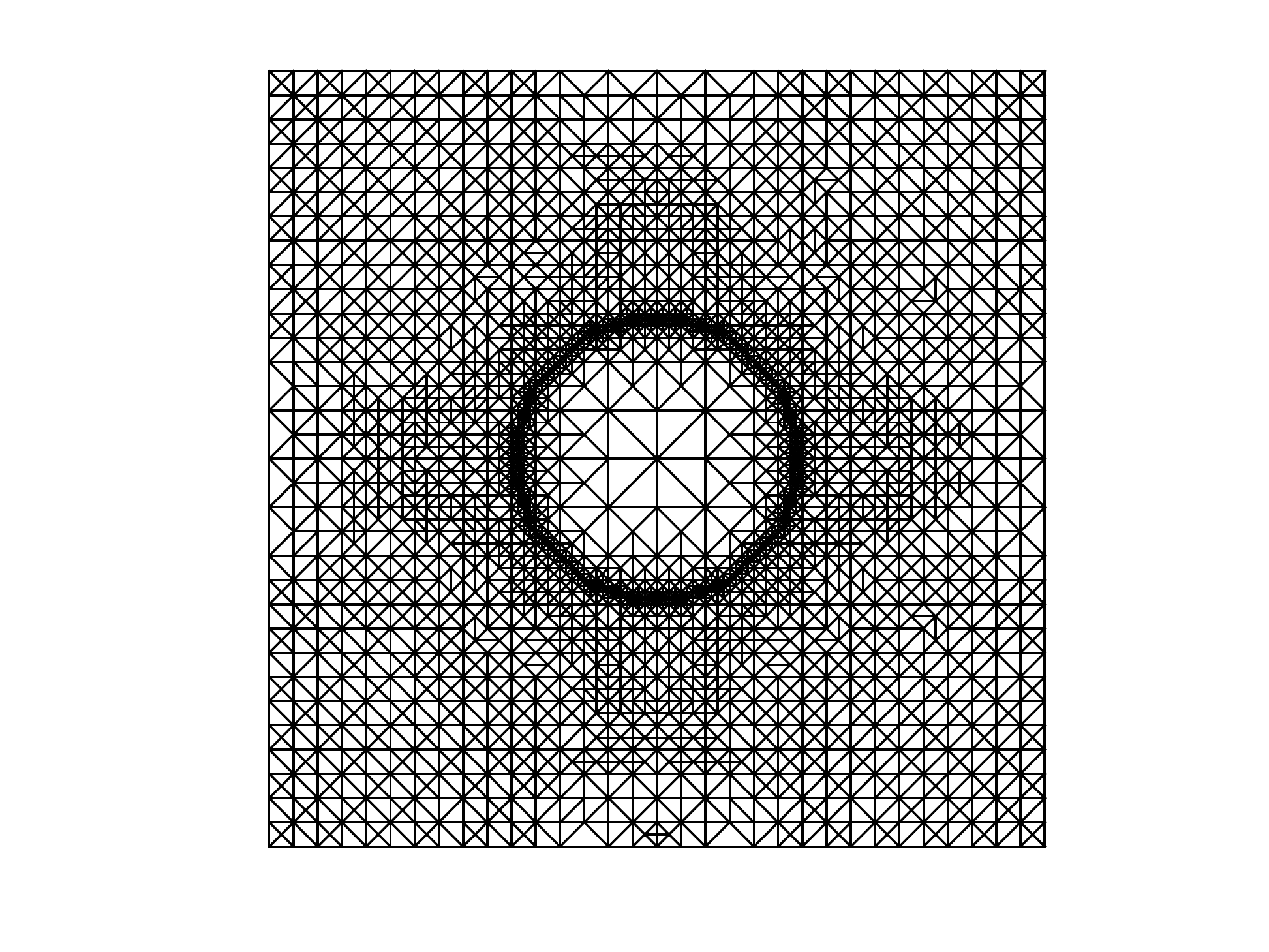}
\end{center}
\caption{Adaptive mesh %with roughly 3000 nodes
 for the cubic $C^0$ interior penalty method for Example~1}
\label{fig:MeshExp2}
\end{figure}
%\goodbreak
%%%%%%%%%%%%%%%%%%%%%%
\subsection{Example 2}\label{subsec:Example2}
 In this example we consider the obstacle problem
 from \cite[Example~4]{BSZZ:2012:Kirchhoff} for a clamped plate occupying the $L$-shaped domain
 $\O=(-0.5,0.5)^2\setminus[0,0.5]^2$  with $f=0$ and
 $\d \psi(x)=1-\Big[\frac{(x_1+1/4)^2}{0.2^2} + \frac{x_2^2}{0.35^2}\Big]$.
 The coincidence set for this problem is presented in
 Figure~\ref{fig:ContactL1}\hspace{1pt}(a).
\begin{figure}[hh]
\subfigure[]
{\includegraphics[width=0.3\linewidth]{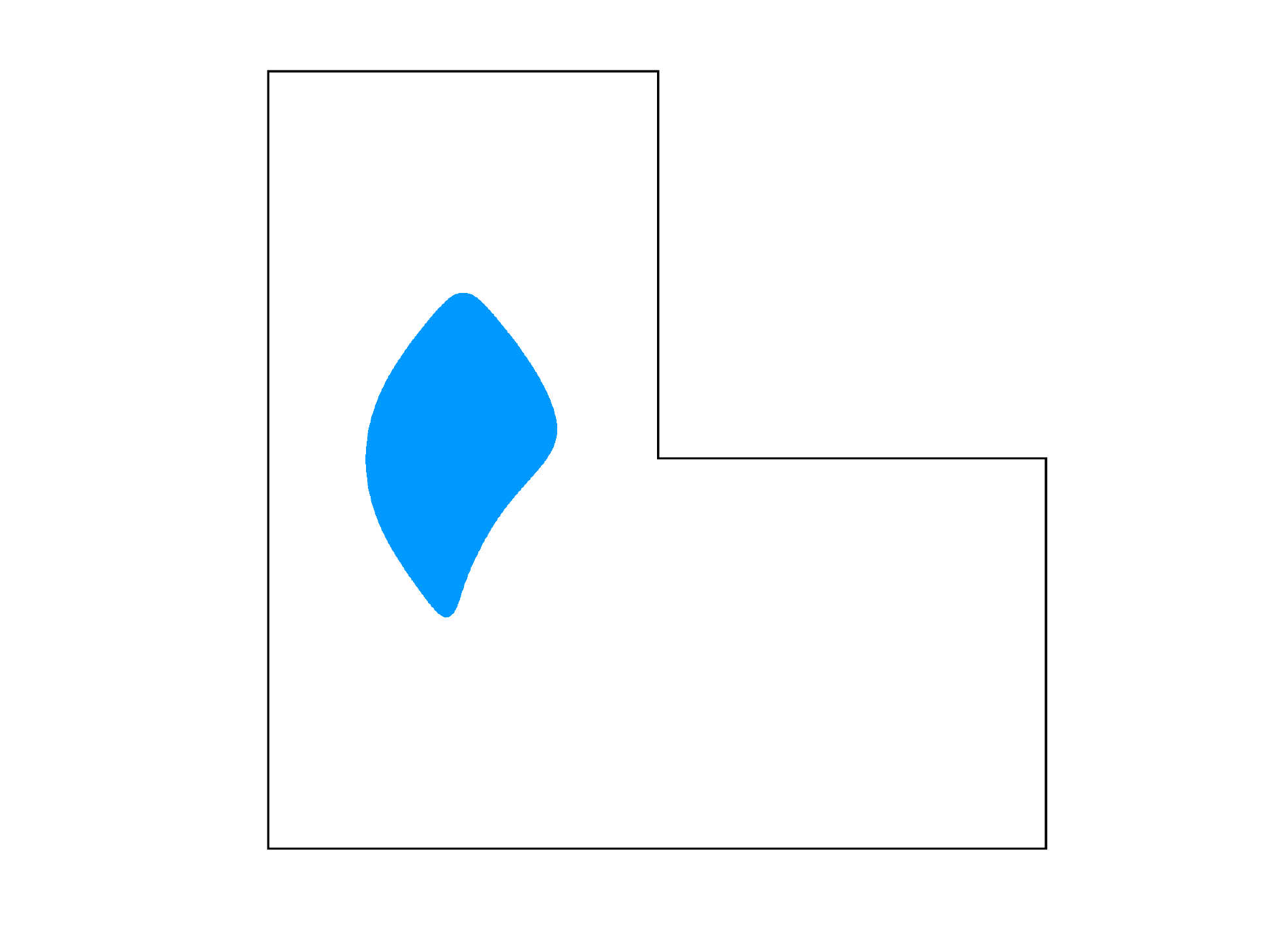}}
\subfigure[]
{\includegraphics[width=0.3\linewidth]{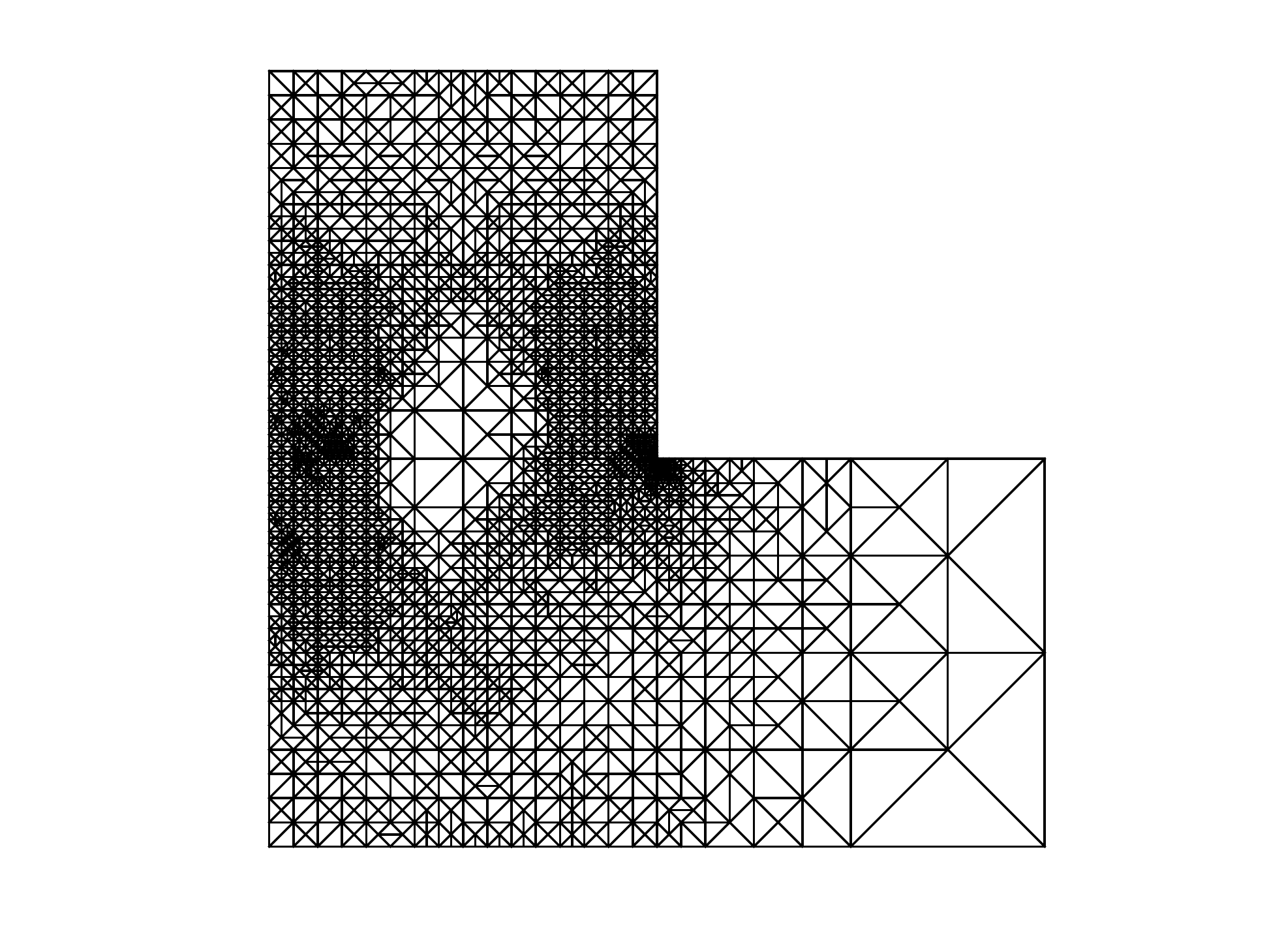}}
\subfigure[]
{\includegraphics[width=0.3\linewidth]{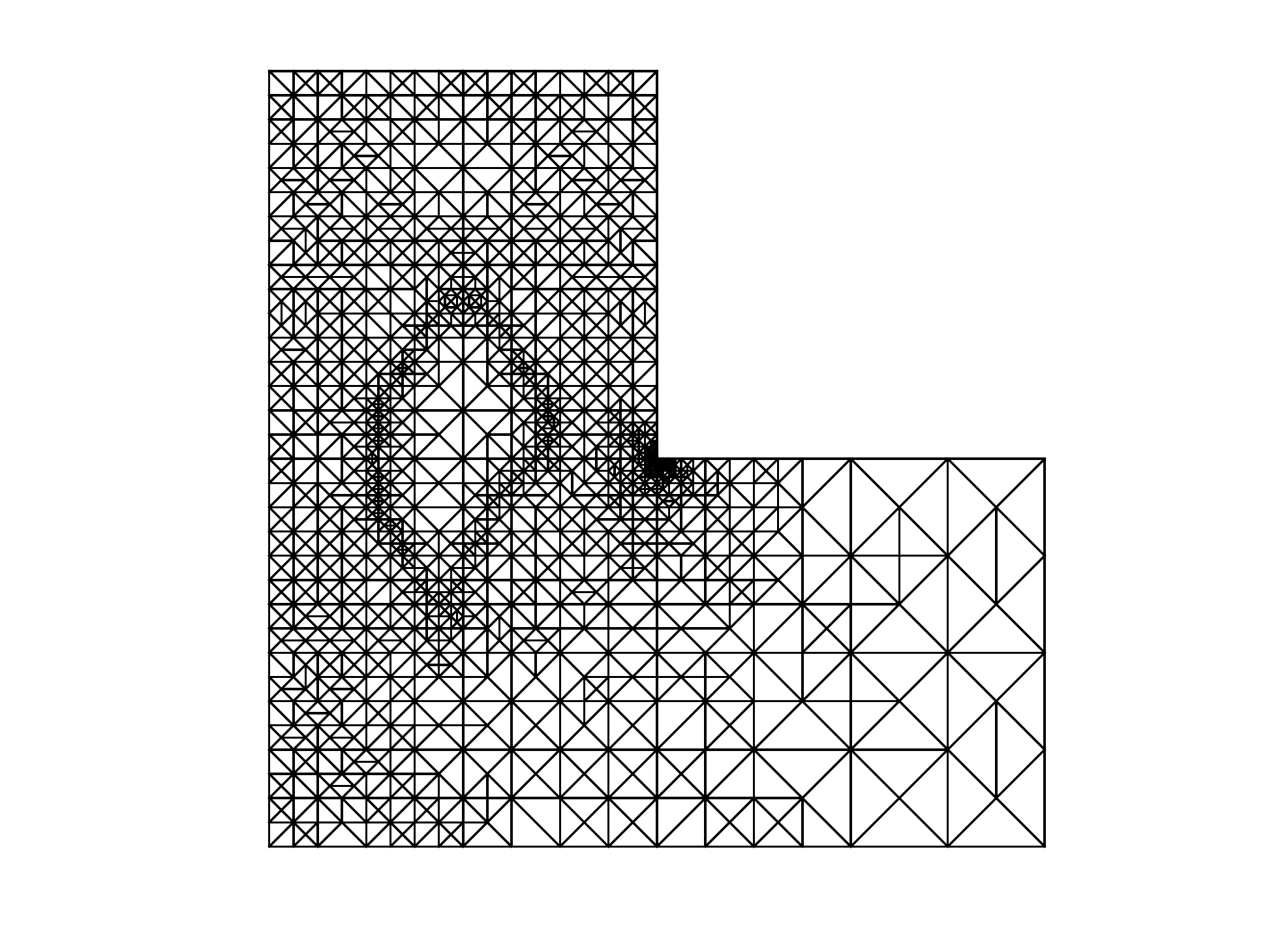}}
 \caption{$L$-shaped domain for Example~2: (a) Coincidence set for the obstacle problem
 (b) Adaptive mesh with  $\approx 3000$ nodes
  for the $P_2$ element
 (c) Adaptive mesh with  $\approx 5000$ nodes for the $P_3$ element}
 \label{fig:ContactL1}
\end{figure}
\par
 In the first experiment
 we solve the discrete obstacle problem with the $P_2$ element
 on uniform and adaptive meshes.   Optimal (resp. suboptimal)
 convergence rate for adaptive (resp. uniform) meshes and the reliability of $\eta_\ell$ are
 observed in
 Figure~\ref{fig:P2Example2}\hspace{1pt}(a), and
 the $O(N_\ell^{-1/2})$ convergence rate of $\|u-u_\ell\|_\ell$ is justified
 by Figure~\ref{fig:P2Example2}\hspace{1pt}(b) and Lemma~\ref{lem:AsymptoticConvergenceRate}.
\begin{figure}[!hh]
  \subfigure[]
   {\includegraphics[width=0.4\linewidth]{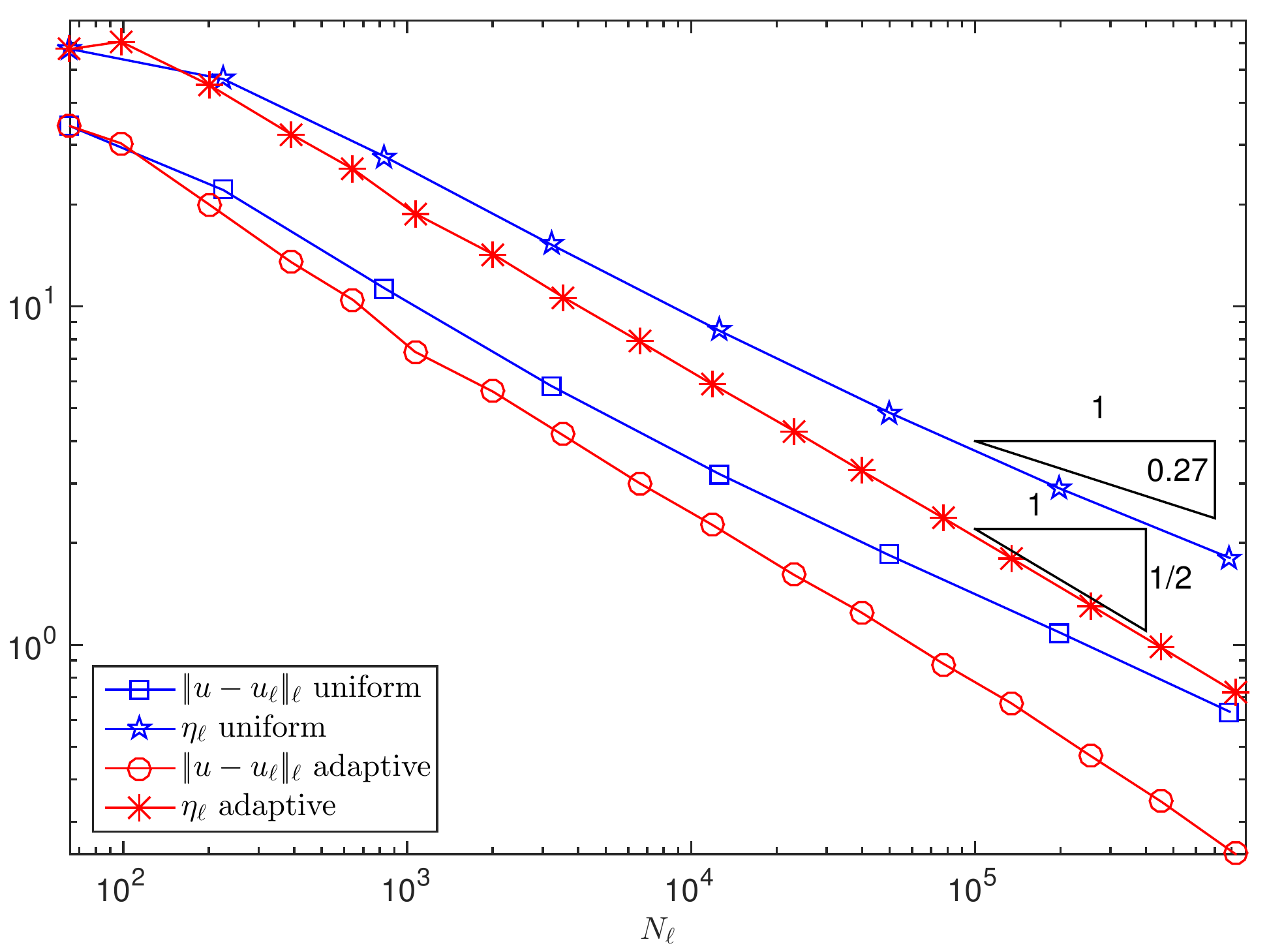}}
  \subfigure[]
  {\includegraphics[width=0.4\linewidth]{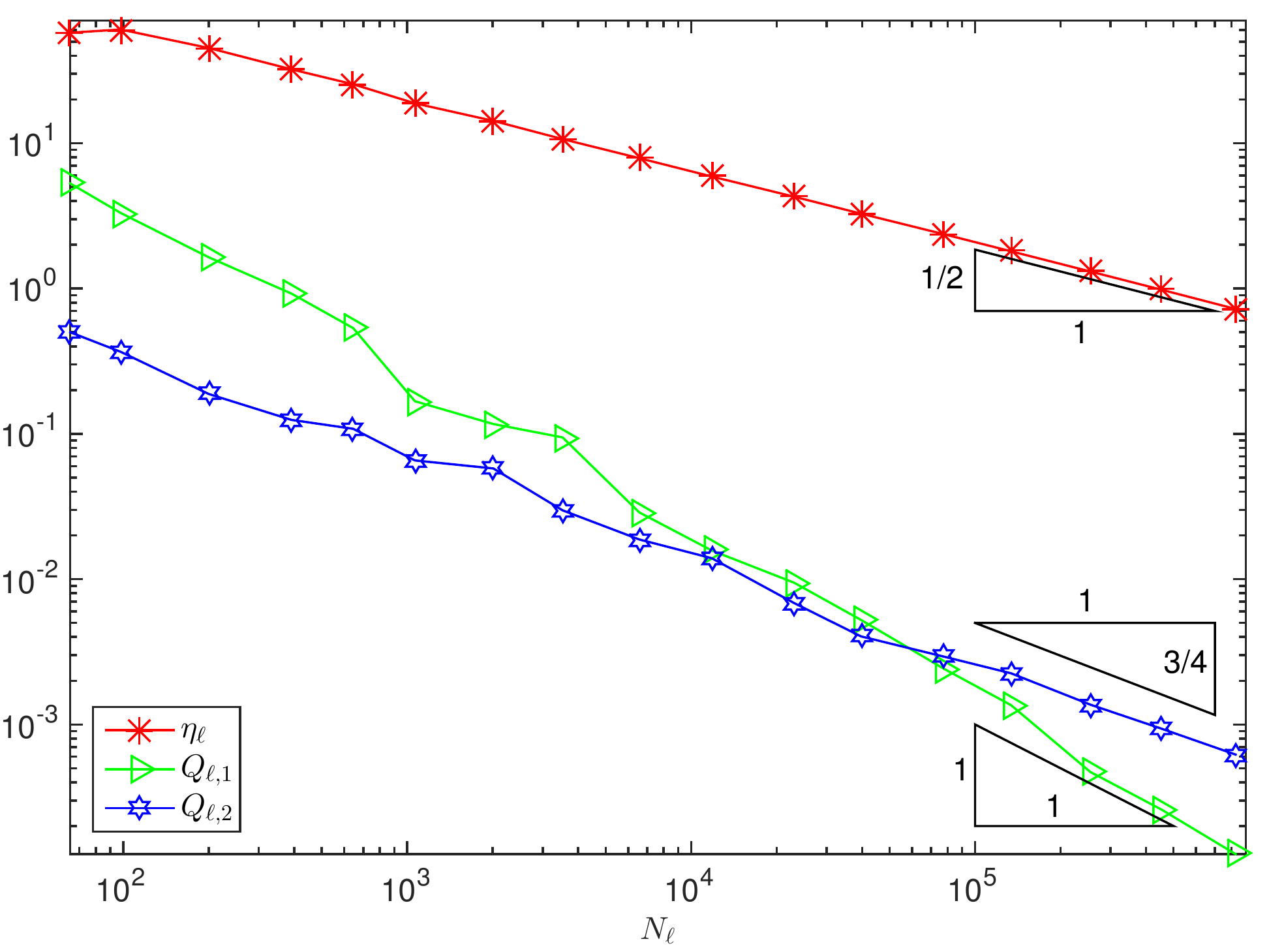}}
\caption{Convergence histories for the quadratic $C^0$ interior penalty method
  for Example~2: (a) $\|u-u_\ell\|_\ell$ and
  $\eta_\ell$, (b)  $\eta_\ell$, $Q_{\ell,1}$ and $Q_{\ell,2}$}
\label{fig:P2Example2}
\end{figure}
\par
 The $O(N_\ell^{-1/2})$ bound for $\Lambda_\ell$ predicted by
 Lemma~\ref{lem:LMTest} and  Figure~\ref{fig:P2Example2}\hspace{1pt}(b)
 is observed in Table~\ref{table:LMP2Example2}, where
 $N_\ell$ increases from $65$ to $827483$.
\begin{table}[hh]
\begin{tabular}{c|c|c|c|c|c|c|c|c|c}
  $\ell$ & $0$ & $1$ &$2$ & $3$ & $4$ & $5$ &$6$ & $7$ & $8$ \\
  \hline
  $\Lambda_\ell N_\ell^{1/2}$ & $2715$ & $391$ & $637$ & $756$ & $1454$ & $654$ & $613$
  & $467$ & $411$  \\
  \hline  \hline
   $\ell$ & $9$ & $10$ &$11$ & $12$ & $13$ & $14$ & $15$ & $16$&  \\
   \hline
   $\Lambda_\ell N_\ell^{1/2}$ & $360$ & $149$ &$255$ &$105$ & $144$ & $70$ & $72$ & $52$
\end{tabular}
\medskip
\caption{$\Lambda_\ell N_\ell^{1/2}$ for the adaptive quadratic $C^0$ interior penalty method for
  Example~2}
 \label{table:LMP2Example2}
\end{table}
\par
 An adaptive mesh with roughly 3000 nodes is displayed in
 Figure~\ref{fig:ContactL1}\hspace{1pt}(b), where
 we observe a strong refinement near the reentrant corner. In contrast the refinement near the
 free boundary is mild.  This is due to the fact that away from the reentrant corner
 the solution belongs to $H^3$ (cf. \cite{Frehse:1971:VarInequality,BR:1980:Biharmonic})
 and we are using the $P_2$ element.
%%%%%%%%%%%%%%%%%%%%%
%
\par
 In the second experiment we solve the obstacle problem
 with the $P_3$ element on
 uniform and adaptive meshes.
  We observe optimal (resp. suboptimal) convergence rate for adaptive (resp. uniform) meshes
  in Figure~\ref{fig:P3Example2}\hspace{1pt}(a) and that $\eta_\ell$ is reliable in both cases.
 Moreover the $O(N_\ell^{-1})$ convergence rate of $\|u-u_\ell\|_\ell$ is
 justified by  Figure~\ref{fig:P3Example2}\hspace{1pt}(b) and Lemma~\ref{lem:AsymptoticConvergenceRate}.
\begin{figure}[!hh]
  \subfigure[]
  {\includegraphics[width=0.4\linewidth]{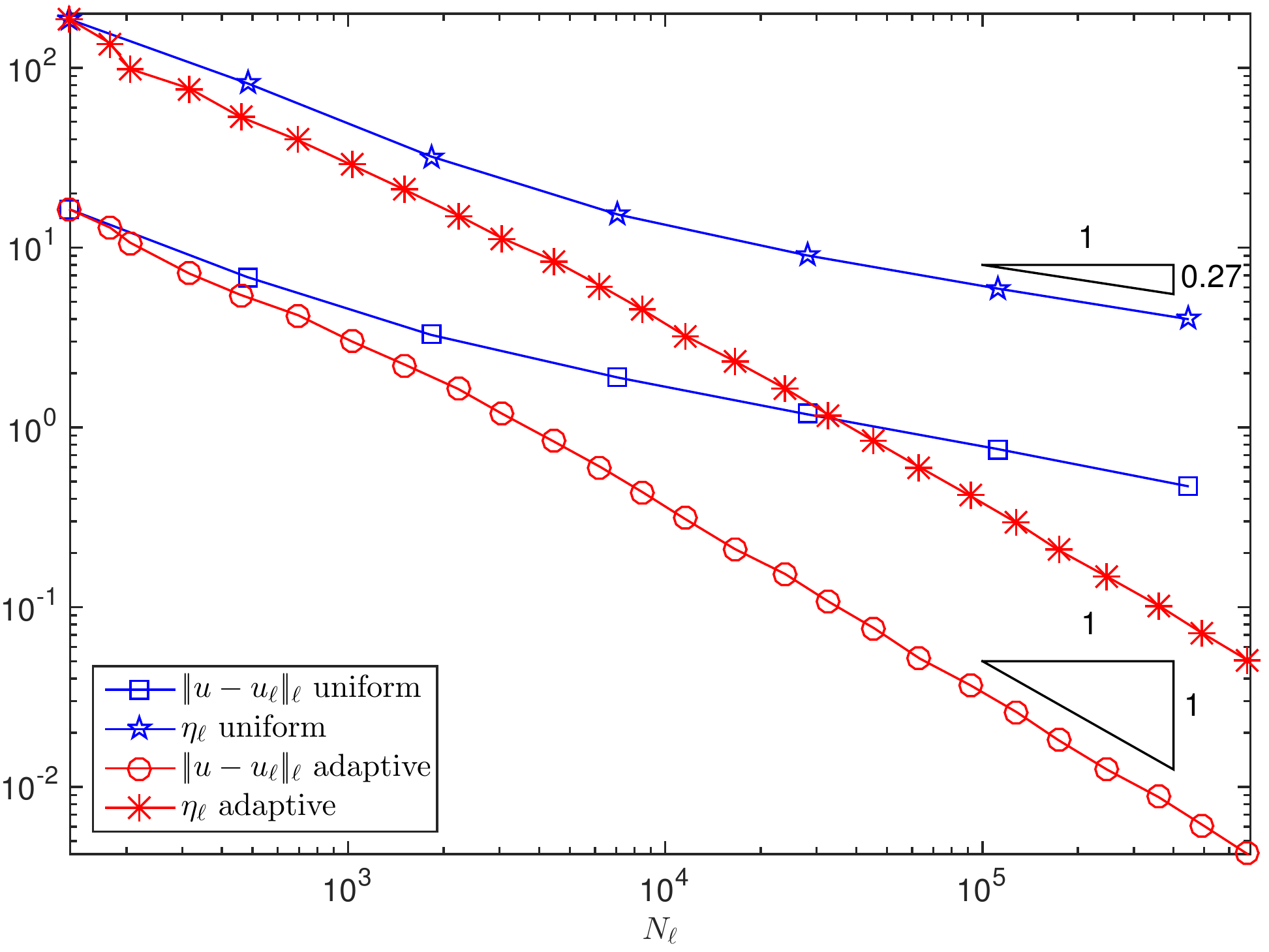}}
  \subfigure[]
  {\includegraphics[width=0.4\linewidth]{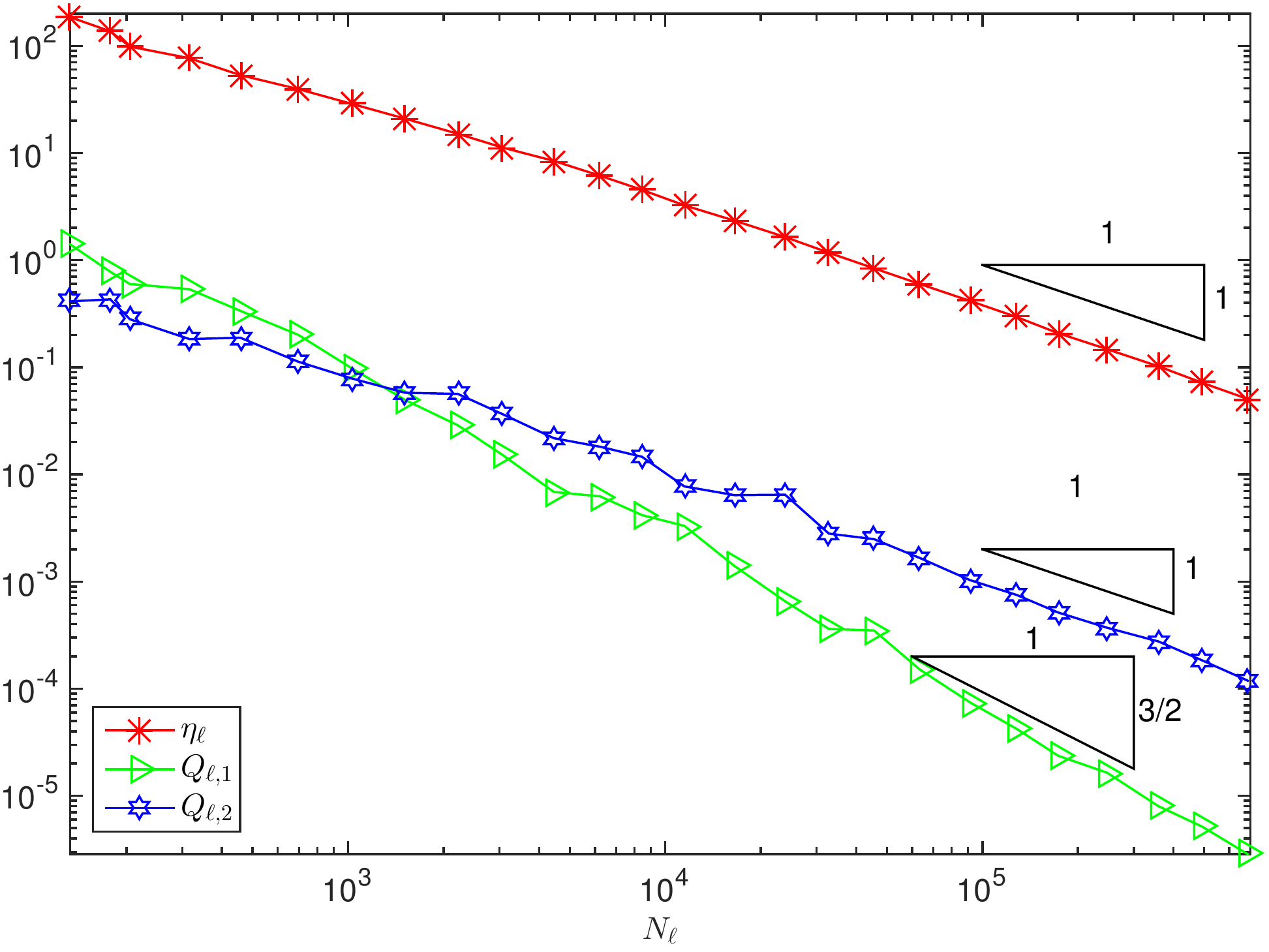}}
\caption{Convergence histories for the cubic $C^0$ interior penalty method
  for Example~2: (a) $\|u-u_\ell\|_\ell$ and
  $\eta_\ell$, (b)  $\eta_\ell$, $Q_{\ell,1}$ and $Q_{\ell,2}$}
\label{fig:P3Example2}
\end{figure}
\par
 The results for $\Lambda_\ell$ are reported in Table~\ref{table:LMP3Example2},
 where the $O(N_\ell^{-1})$ bound for $\Lambda_\ell$ predicted by
 Lemma~\ref{lem:LMTest} and Figure~\ref{fig:P3Example2}\hspace{1pt}(b) can be observed.
 Note that there are large oscillations at the beginning before the coincidence has been
 captured by the adaptive mesh.  Here $N_\ell$ increases from
 $N_0=133$ to $N_{22}=358792$.
\begin{table}[hh]
\begin{tabular}{c|c|c|c|c|c|c|c|c}
  $\ell$ & $0$ & $1$ &$2$ & $3$ & $4$ & $5$ &$6$ & $7$  \\
  \hline
  $\Lambda_\ell N_\ell$ & $15398$ & $16877$ & $2893$ & $1035$ & $11806$& $8925$ & $15493$ & $5993$
   \\
  \hline  \hline
   $\ell$ &  $8$ & $9$ & $10$ & $11$ & $12$ &$13$ & $14$ & $15$\\
   \hline
   $\Lambda_\ell N_\ell$ &  $1048$ & $5162$ & $3544$ & $3271$ & $1362$ &$119$ & $580$ & $778$ \\
   \hline\hline
   $\ell$ &  $16$ & $17$ & $18$& $19$ &$20$  & $20$ &$21$ &$22$\\
   \hline
   $\Lambda_\ell N_\ell$ & $77$ & $96$ & $68$ & $92$ & $147$ & $116$ &$754$ & $885$
\end{tabular}
\medskip
\caption{$\Lambda_\ell N_\ell$ for the adaptive cubic $C^0$ interior penalty method for
  Example~2}
 \label{table:LMP3Example2}
\end{table}

\par
 An adaptive mesh with roughly 5000 nodes is displayed in Figure~\ref{fig:ContactL1}\hspace{1pt}(c),
  where we observe strong refinement near both the reentrant corner and the free boundary.
\goodbreak
%%%%%%%%%%%%%%%%%%%%%%%%%%%%%%%%%%%%%%%%%%%
\subsection{Example 3}\label{subsec:Example3}
 In this example we consider the obstacle problem on the $L$-shaped domain
  $\O=(-0.5,0.5)^2\setminus[0,0.5]^2$ with
 $$
\psi(x)= - [\sin(2 \pi (x_1+0.5) (x_2+0.5)) \sin(4 \pi (x_1-0.5) (x_2-0.5))] - 0.35
 $$
 and
\begin{equation*}
f(x)=
\begin{cases}
10^3 \Big( \frac 12 e^{(x_1+0.25)^2+(x_2+0.25)^2} \Big) &\qquad x_1\leq 0, x_2 > 0 \\
0  &\qquad x_1 \leq 0, x_2 \leq 0 \\
10^3 \Big( \frac{1}{2} + [(x_1-0.25)^2+(x_2+0.25)^2]^{3/2} \Big) &\qquad x_1 \geq 0, x_2 \leq 0
\end{cases}\,.
\end{equation*}
 For this example, the coincidence set is one dimensional
 (cf. Figure~\ref{fig:ContactSetL2}\hspace{1pt}(a)).
\begin{figure}[hh]%[!htb]
%\begin{center}
\subfigure[]
   {\includegraphics[width=0.3\linewidth]{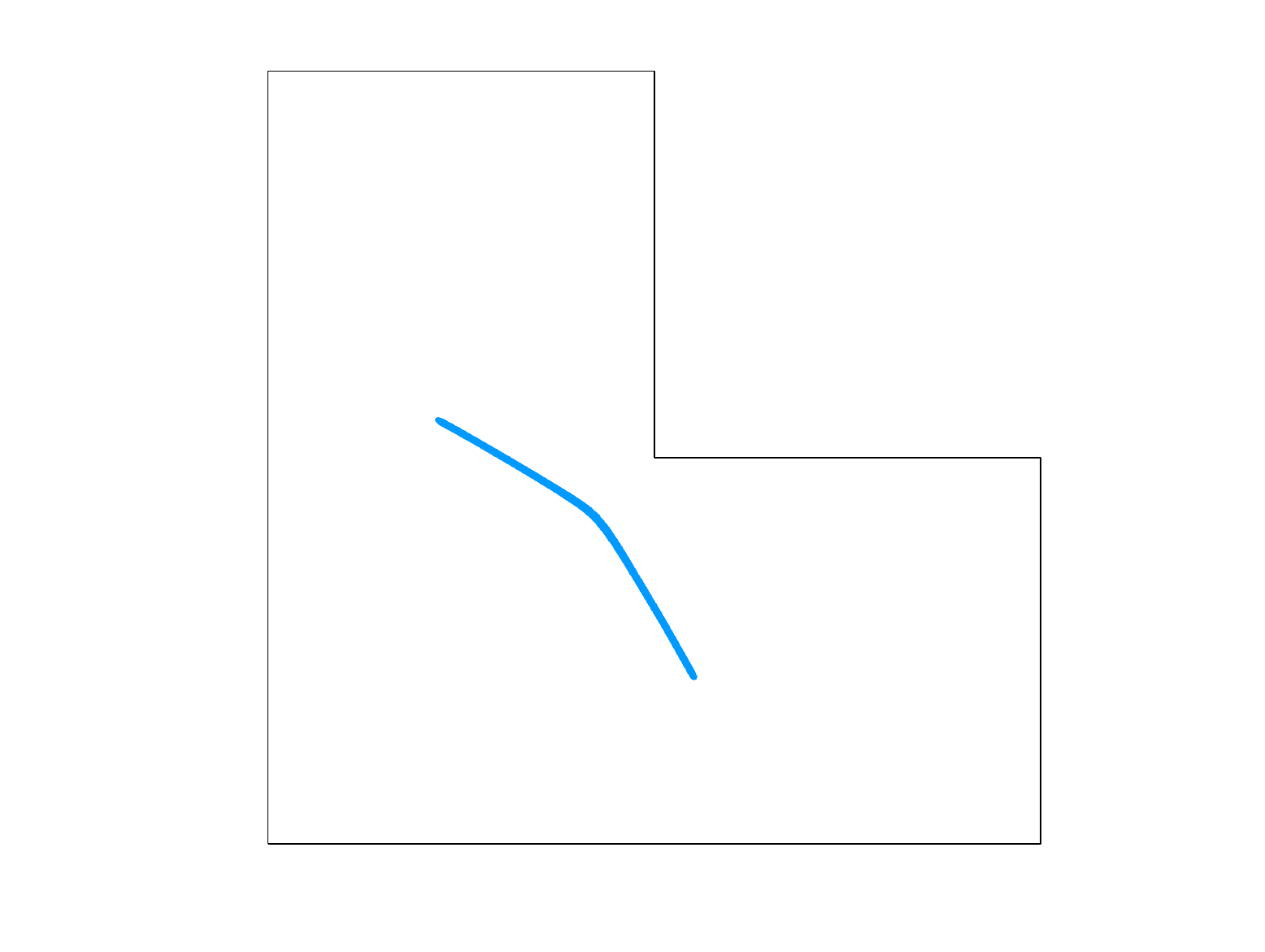}}
\subfigure[]
  {\includegraphics[width=0.3\linewidth]{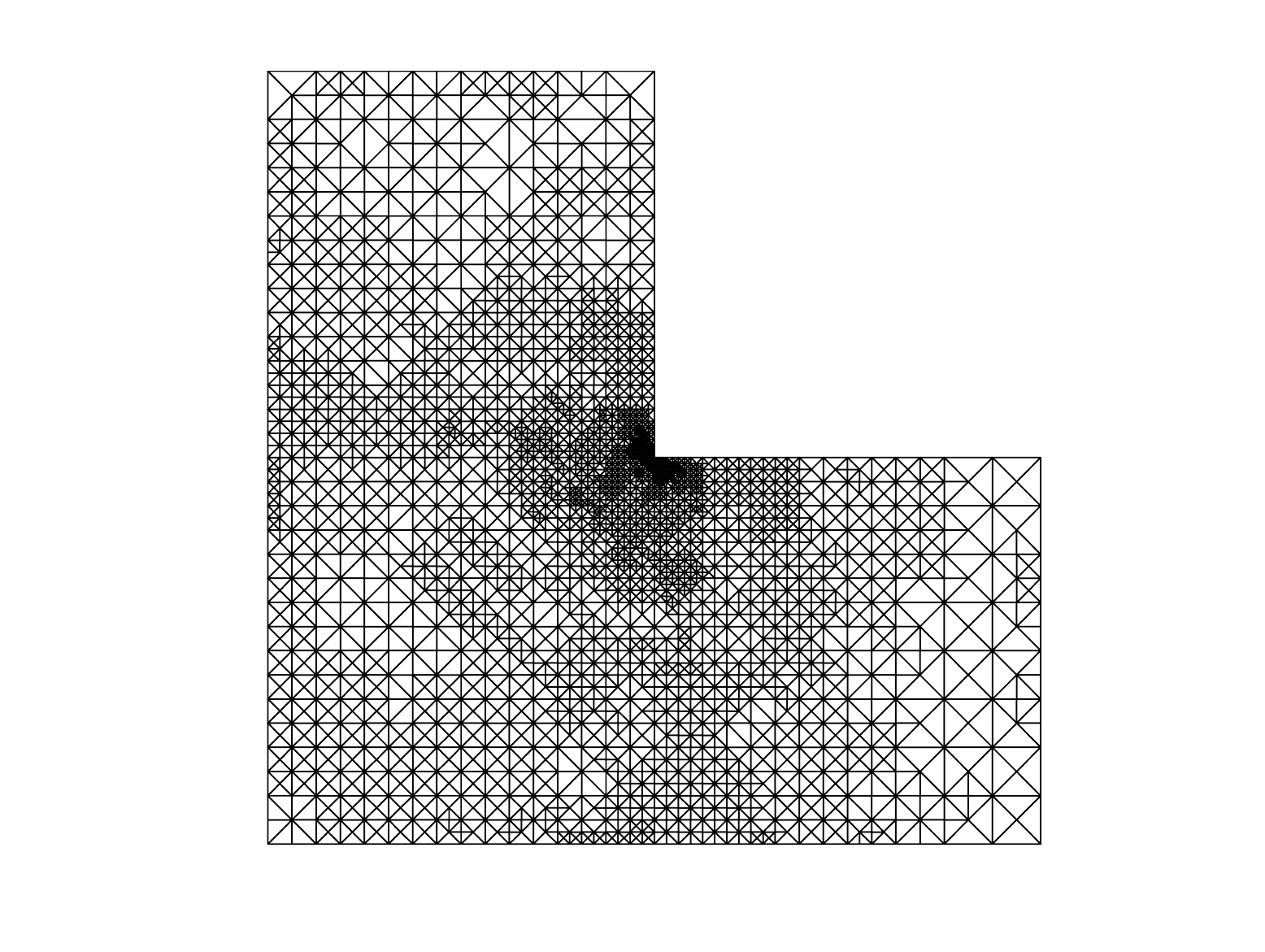}}
\subfigure[]
  {\includegraphics[width=0.3\linewidth]{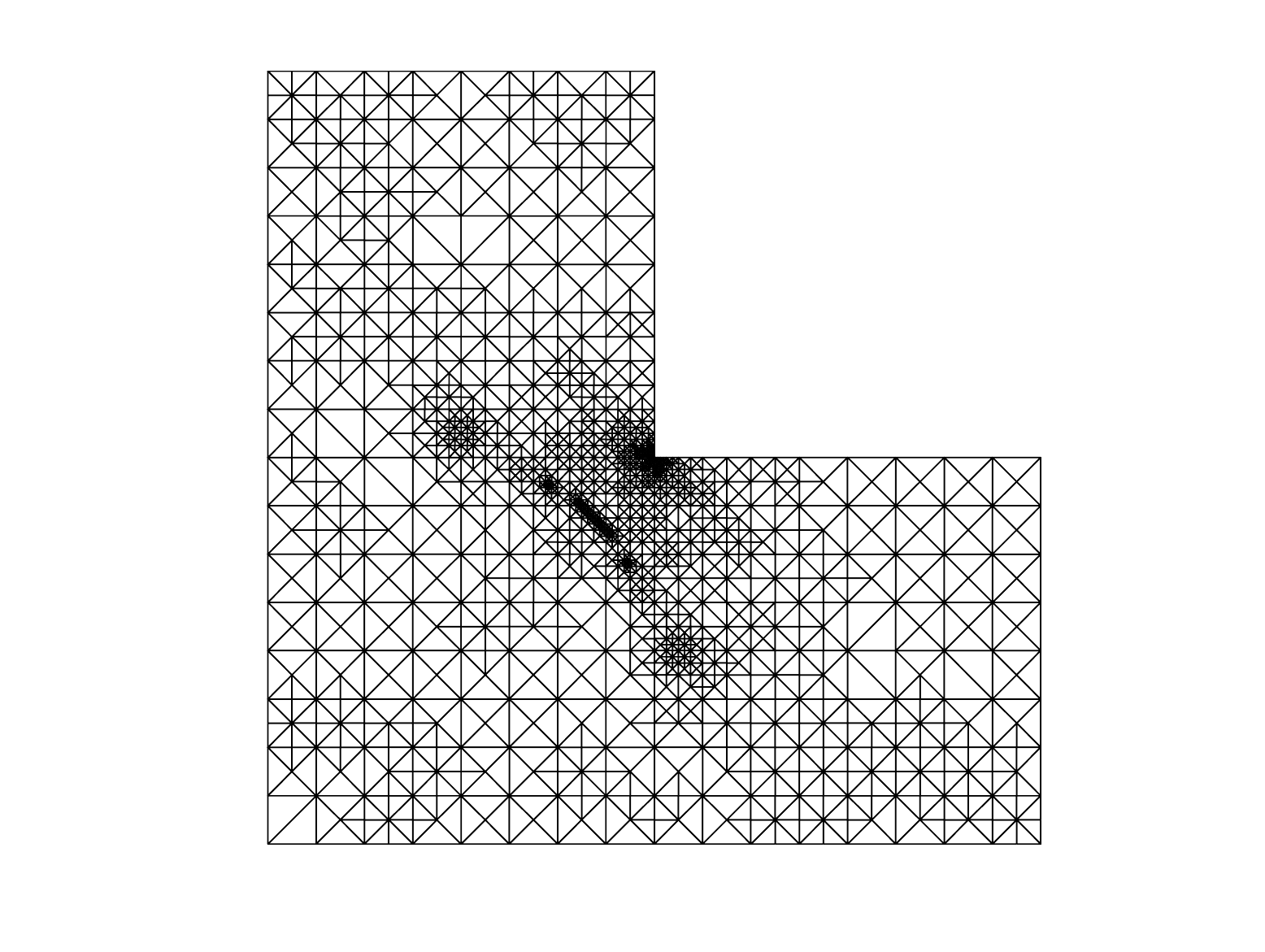}}
 \caption{$L$-shaped domain for Example~3: (a) Coincidence set for the obstacle problem
 (b) Adaptive mesh with  11062 dof
  for the $P_2$ element
 (c) Adaptive mesh with  12841 dof for the $P_3$ element}
\label{fig:ContactSetL2}
\end{figure}
\par
 In the first experiment we solve the obstacle problem with the $P_2$
  element on uniform and adaptive meshes.
  We observe optimal (resp. suboptimal) convergence rate for adaptive (resp. uniform) meshes
 in Figure~\ref{fig:P2Example3}\hspace{1pt}(a)
  and also the reliability of $\eta_\ell$.
  The $O(N_\ell^{-1/2})$ convergence rate of $\|u-u_\ell\|_\ell$ is confirmed by
  Figure~\ref{fig:P2Example3}\hspace{1pt}(b) and
 Lemma~\ref{lem:AsymptoticConvergenceRate}.
\begin{figure}[!hh]%[!htb]
\subfigure[]
   {\includegraphics[scale=0.35]{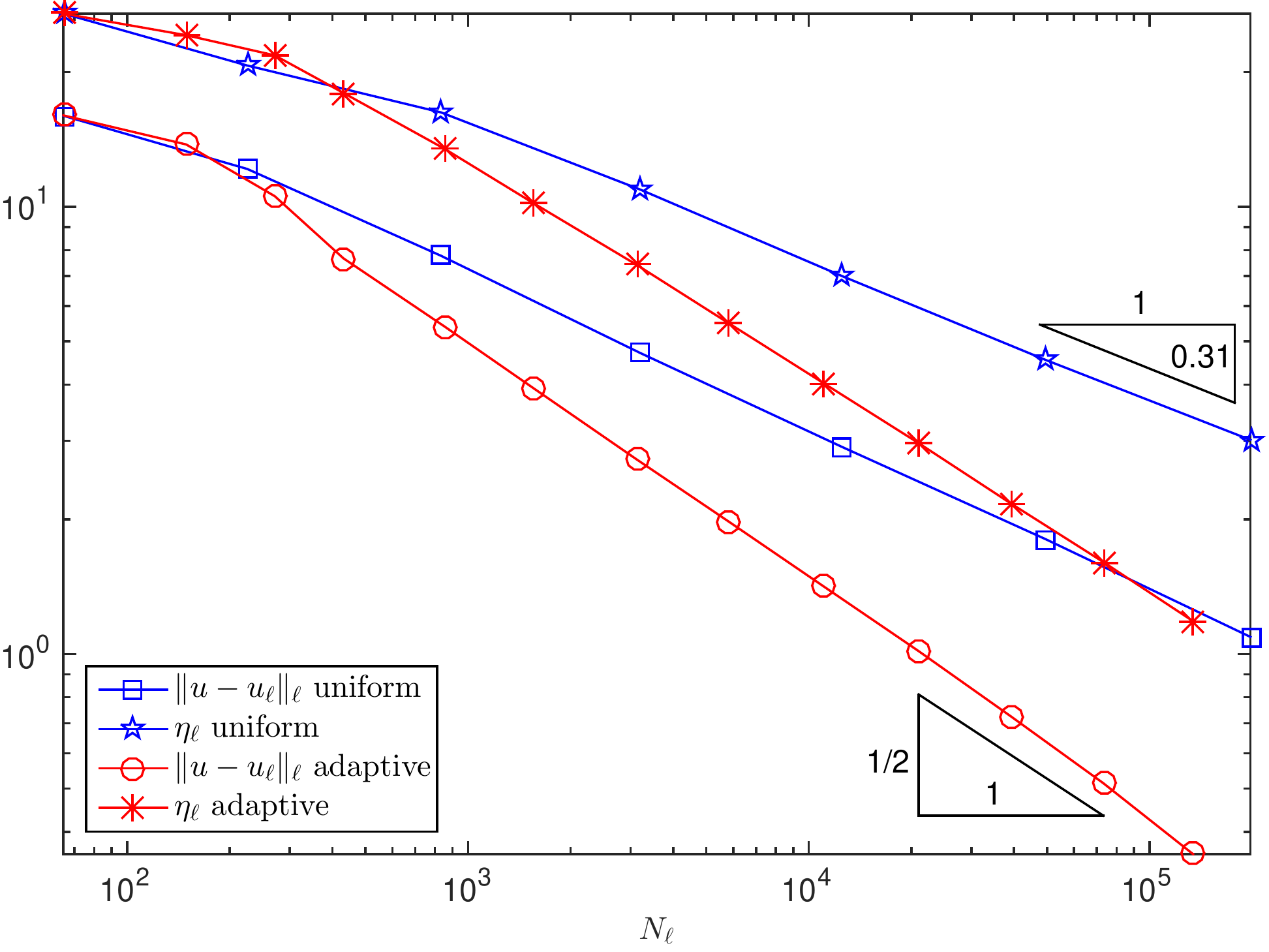}}
\subfigure[]
   {\includegraphics[scale=0.35]{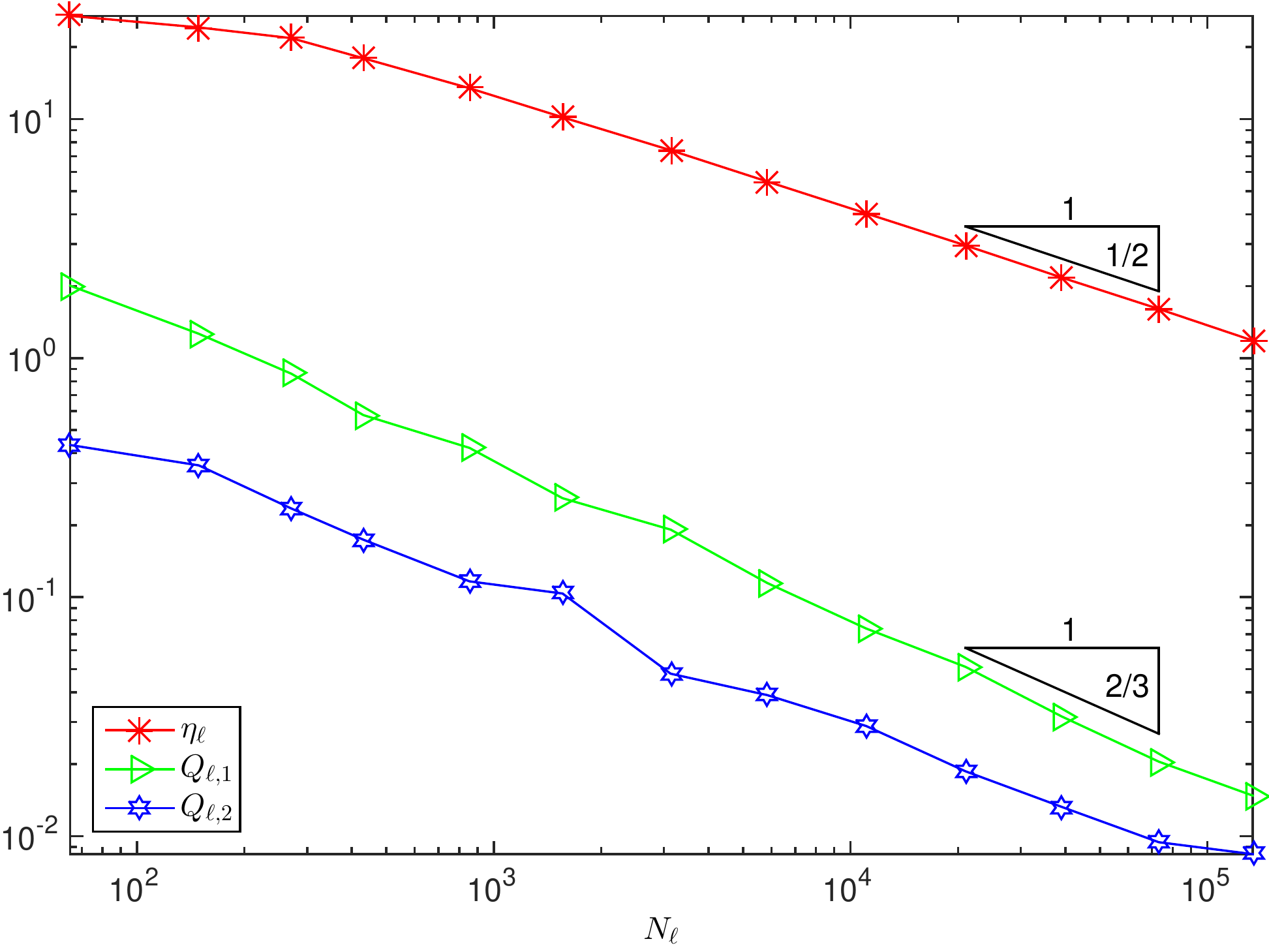}}
\caption{Convergence histories for the quadratic $C^0$ interior penalty method
  for Example~3: (a) $\|u-u_\ell\|_\ell$ and
  $\eta_\ell$, (b)  $\eta_\ell$, $Q_{\ell,1}$ and $Q_{\ell,2}$}
\label{fig:P2Example3}
\end{figure}
\par
 The results in Table~\ref{table:LMP2Example3} agrees with the
  $O(N_\ell^{-1/2})$ bound for $\Lambda_\ell$ that follows from
 Lemma~\ref{lem:LMTest} and Figure~\ref{fig:P2Example3}\hspace{1pt}(b).
 The number of dof increases from $N_0=65$ to
  $N_{12}=134096$.
\begin{table}[hh]
\begin{tabular}{c|c|c|c|c|c|c|c |c |c |c |c |c |c}
  $\ell$ & $0$ & $1$ &$2$ & $3$ & $4$ & $5$ &$6$ & $7$ & $8$ & $9$& $10$ &$11$ & $12$  \\
  \hline
  &&&&&&&&&&&&& \\[-11pt]
  $\Lambda_\ell N_\ell^{1/2}$ & $1151$ & $501$ & $92$ & $201$ & $120$ & $419$
   &$201$    &$98$ & $75$ & $34$ & $76$ &$40$ & $36$\\
\end{tabular}
\medskip
\caption{$\Lambda_\ell N_\ell^{1/2}$ for the adaptive quadratic $C^0$ interior penalty method for
  Example~3}
 \label{table:LMP2Example3}
\end{table}
\par
 An adaptive mesh with 11062 dof is depicted in
 Figure~\ref{fig:ContactSetL2}\hspace{1pt}(b), where
 we observe that the only strong refinement is around the reentrant corner.
  This is again due to the fact that
 away from the reentrant corner the solution belongs to $H^3$ and we are using the $P_2$ element.
\par
 In the second experiment we solve the obstacle problem
 with the $P_3$ element on
 uniform and adaptive meshes.
  We observe optimal (resp. suboptimal) convergence rate for adaptive (resp. uniform) meshes
 in Figure~\ref{fig:P3Example3}\hspace{1pt}(a) and also the reliability
  of $\eta_\ell$.  Furthermore the $O(N_\ell^{-1})$ convergence rate
  for $\|u-u_\ell\|_\ell$ is justified by Figure~\ref{fig:P2Example3}\hspace{1pt}(b) and
  Lemma~\ref{lem:AsymptoticConvergenceRate}.
\begin{figure}[!hh]%[!htb]
\subfigure[]
   {\includegraphics[scale=0.35]{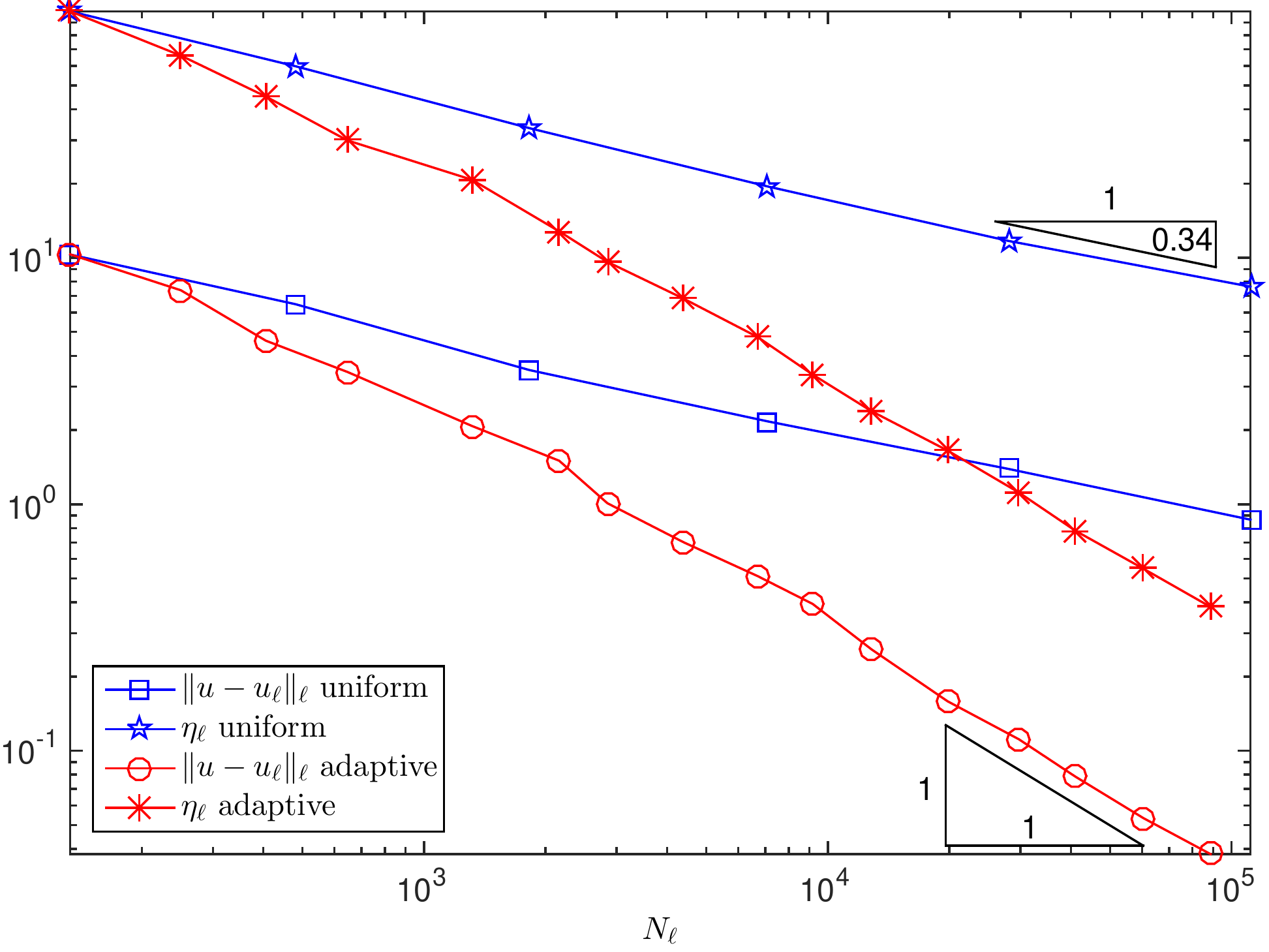}}
\subfigure[]
   { \includegraphics[scale=0.35]{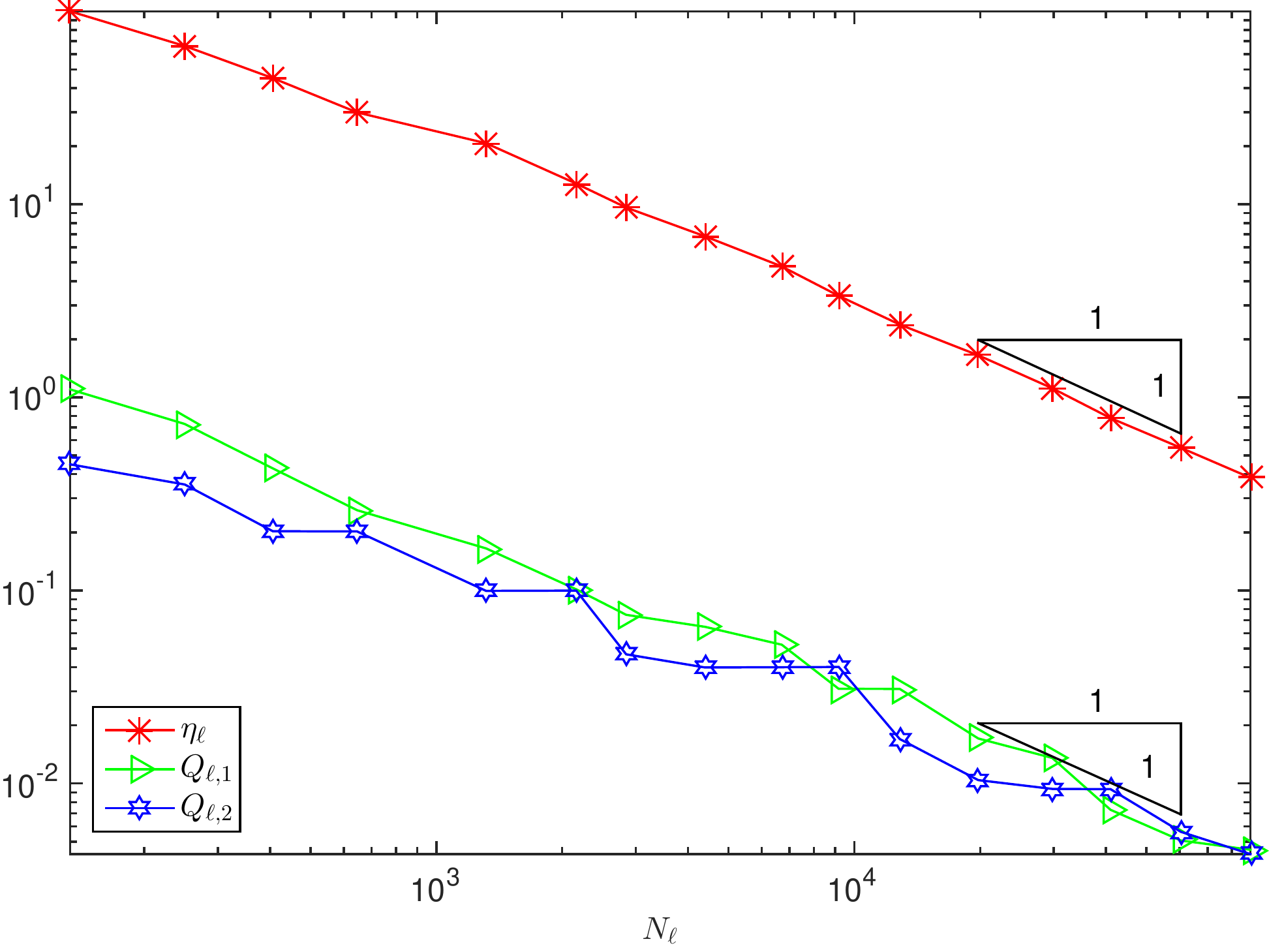}}
\caption{Convergence histories for the cubic $C^0$ interior penalty method
  for Example~3: (a) $\|u-u_\ell\|_\ell$ and
  $\eta_\ell$, (b)  $\eta_\ell$, $Q_{\ell,1}$ and $Q_{\ell,2}$}
\label{fig:P3Example3}
\end{figure}
\par
 The results in Table~\ref{table:LMP3Example3} agrees with
 the $O(N_\ell^{-1})$ bound for $\Lambda_\ell$ predicted by
 Lemma~\ref{lem:LMTest} and Figure~\ref{fig:P3Example3}\hspace{1pt}(b).
 Here $N_\ell$ increases from $N_0=133$ to $N_{15}=88699$.
\begin{table}[hh]
\begin{tabular}{c|c|c|c|c|c|c|c|c}
  $\ell$ & $0$ & $1$ &$2$ & $3$ & $4$ & $5$ &$6$ & $7$  \\
  \hline
  $\Lambda_\ell N_\ell$ & $11724$ &$ 1782$ &$ 1842$ & $32888$
  & $1046$ & $6439$ & $2974$ & $2588$
   \\
  \hline  \hline
   $\ell$ &  $8$ & $9$ & $10$ & $11$ & $12$ &$13$ & $14$ & $15$\\
   \hline
   $\Lambda_\ell N_\ell$ & $2781$ & $25657$ & $2215$& $3805$ & $5177$
   & $2030$ & $1092$ & $2355$  \\
\end{tabular}
\medskip
\caption{$\Lambda_\ell N_\ell$ for the adaptive cubic $C^0$ interior penalty method for
  Example~3}
 \label{table:LMP3Example3}
\end{table}
\par
 An adaptive mesh with 12841 dof is depicted in
 Figure~\ref{fig:ContactSetL2}\hspace{1pt}(c), where
 we observe strong refinement  around the reentrant corner and the coincidence set.

%%%%%%%%%%%%%%%%%%%%%%%%%%%%%%%%%%%%%%%%%%%
\section{Conclusions}\label{sec:Conclusions}
%%%%%%%%%%%%%%%%%%%%%%%%%%%%%%%%%%%%%%%%%%%
 We have developed a simple {\em a posteriori} error analysis of $C^0$ interior penalty methods
 for the displacement obstacle problem of clamped
 Kirchhoff plates by taking advantage of the fact that the Lagrange multiplier for the
 discrete problem can be represented naturally as the sum of Dirac point measures supported
 at the vertices of the triangulation.
   Numerical results indicate that the adaptive
 algorithm based on a standard {\em a posteriori} error estimator originally
 developed for boundary value problems also
 performs optimally  for quadratic and  cubic $C^0$ interior penalty methods
 for obstacle problems.
 However the theoretical justification of convergence and optimality for adaptive
 $C^0$ interior penalty methods remains open even in the case
 when the obstacle is absent.
\par
 The results in this paper can be extended to the displacement obstacle problem of
 the biharmonic equation with the boundary conditions of simply
 supported plates or the Cahn-Hilliard type.  In the case where $\O$ is convex, such
 problems are related to distributed elliptic optimal control problems
  with pointwise state constraints
  \cite{LGY:2009:Control,GY:2011:State,BSZ:2013:OptimalControl,BSZ:2015:PP} and
  can also be considered in three dimensional domains.
  Adaptive finite element methods for these problems based on the approach in this paper
  are ongoing projects.
%
%%%%%%%%%%%%%%%%%%%%%%%%%%%%%%%%%%%%%%%%%%

\end{document}